\documentclass[11pt]{article}

\usepackage[margin=2.5cm]{geometry}

\usepackage[utf8]{inputenc}
\usepackage{fourier}
\usepackage{microtype}
\usepackage{eurosym}

\usepackage{amsmath,amsfonts,amssymb,amsthm,mathtools}

\usepackage{authblk}

\usepackage{xcolor}

\usepackage{tikz}
\usepackage{pgfplots}
\pgfplotsset{compat=1.14}

 \usepackage[style=alphabetic,
   safeinputenc,
   url=false,giveninits=true,
   sorting=nyt,
   maxnames=5]%
   {biblatex}
 \renewbibmacro{in:}{
   \ifentrytype{article}{}{\printtext{\bibstring{in}\intitlepunct}}}
 \DeclareFieldFormat
   [article,inbook,incollection,inproceedings,patent,unpublished,online]
   {title}{#1}
 \DeclareFieldFormat
   [thesis]
   {title}{\mkbibemph{#1}}

\usepackage[colorlinks]{hyperref}
\usepackage{mdframed}

\usepackage{pgffor}

\newtheorem{thm}{Theorem}[section]%
\newtheorem{cor}[thm]{Corollary}%
\newtheorem{prop}[thm]{Proposition}%
\newtheorem{lem}[thm]{Lemma}%
\newtheorem{defi}[thm]{Definition}%
\theoremstyle{remark}
\newtheorem{rem}{Remark}%
\newcommand{\prfstep}[1]{\par\textbf{#1.} \hspace{2em}}

\newcommand{\DeclareAlphabet}[2]{
  \foreach \x in {A,B,...,Z}{%
    \expandafter\xdef
    \csname #1\x\endcsname{%
      \noexpand#2{\x}}%
  }
}

\DeclareAlphabet{d}{\mathbb}  
\DeclareAlphabet{c}{\mathcal}
\DeclareAlphabet{r}{\mathrm}
\DeclareAlphabet{b}{\mathbf}

\newcommand{\ABS}[1]{{{\left\vert#1\right\vert}}} 
\let\abs\ABS
\newcommand{\NRM}[1]{{{\left\Vert#1\right\Vert}}} 
\newcommand{\PAR}[1]{{{\left(#1\right)}}} 
\newcommand{\prb}[2][]{\dP_{#1}\left[#2\right]}
\newcommand{\esp}[2][]{\dE_{#1}\left[#2\right]}

\newcommand{\ind}[1]{\mathbf{1}_{#1}}

\newcommand{\defEgal}{\mathrel{\mathop:}=}
\newcommand{\egalDef}{=:}

\DeclareMathOperator{\Lip}{Lip}

\title{Convergence of metadynamics: discussion of the adiabatic hypothesis}

\author[1]{Benjamin Jourdain}
\author[1]{ Tony Lelièvre} %
\author[2]{Pierre-André Zitt} %
\affil[1]{CERMICS, \'Ecole des Ponts, Universit\'e  Paris-Est, INRIA, 77455 Marne-La-Vallée, France.}
\affil[2]{LAMA UMR 8050, CNRS-Universit\'e  Paris-Est-Marne-La-Vallée,
  77454 Marne-La-Vallée, France.}

\date{\today}

\begin{document}

\maketitle

\abstract{By drawing a parallel between metadynamics and self
  interacting models for polymers, we study the longtime convergence
  of the original metadynamics algorithm in the adiabatic setting,
  namely when the dynamics along the collective variables decouples
  from the dynamics along the other degrees of freedom. We
  also discuss the bias which is introduced when the
  adiabatic assumption does not holds.}


\section*{Introduction}
\addcontentsline{toc}{section}{Introduction}

The objective of this work is to discuss the long time properties of
the metadynamics algorithm~\cite{laio-parrinello-02}, by
exploiting in particular similarities between metadynamics and self
interacting stochastic models for polymers~\cite{TV11,tarres-toth-valko-12}.

The metadynamics algorithm consists in evolving a random walker and
adding a penalty over already
visited states in order to escape from metastable states. More precisely,
one considers a process $(Q_t)_{t \ge 0}$ following originally a
dynamics with invariant probability measure $\pi(dq)=\exp(-\beta V(q))
\, dq$ (where $V$ is the potential energy function and $\beta>0$ is the
inverse temperature). One is also given a so-called collective
variable $\xi$ which is a function of the position $q$ giving a low-dimension description of the state of the system (say for
simplicity that $\xi$ is with values in the one-dimensional torus). An important quantity related to the target distribution
$\pi$ and the collective variable $\xi$ is the so-called {\em free
energy}~\cite{lelievre-rousset-stoltz-book-10} defined by 
\begin{equation}\label{eq:FE}
\exp(-\beta F(z)) dz = \xi_\# \pi
\end{equation}
where $\xi_\# \pi$ denotes the image of the measure $\pi$ by $\xi$.
As the
walker $(Q_t)_{t \ge 0}$ evolves, the dynamics is modified in order to target at time $t$ a biased
measure $\pi_t$, where
\begin{equation}\label{eq:penalize}
\frac{d\pi_t}{d\pi}(q) \varpropto\exp\left(- \beta \int_0^t \delta^\epsilon
  \left(\xi(Q_s)-\xi(q)\right) \, ds\right).
\end{equation}
and $\delta^\epsilon$ is a smooth approximation of the Dirac delta
function. 
The aim of metadynamics is
twofold: (i) to converge more quickly to an equilibrium state than for
the original unbiased dynamics; (ii) to obtain in the long-time limit
an estimate of the free energy as the opposite of the penalty
term $\Psi_t:z \mapsto \int_0^t \delta^\epsilon
\left(\xi(Q_s)-z\right) \, ds$. The function $\Psi_t$ records the already
visited values of $\xi$: intuitively, one could think of $\Psi_t$ as
``computational sand filling the free energy wells''. Concerning the long-time convergence, if one keeps on penalizing the already
visited states with the same penalization strength all along the trajectory, as in~\eqref{eq:penalize}, the
penalty term $\Psi_t$ cannot converge. Two ideas have been proposed in the
literature to overcome this difficulty. The first idea is {\em averaging}, see for
example~\cite[Section 3.1]{laio-gervasio-08}: this consists
in considering the long-time limit of the time average of
the penalizing term:  $\frac{1}{t} \int_0^t \Psi_s(z) \, ds$. The
second idea is {\em vanishing adaption}: this consists in reducing the
penalization strength to zero in the long-time limit, namely add a
multiplicative factor $\alpha(s)$ in front of $\delta^\epsilon
  \left(\xi(Q_s)-\xi(q)\right)$ in~\eqref{eq:penalize}, where $\alpha(s)$
  goes to $0$ as $s \to \infty$. The so-called
{\em well-tempered metadynamics}~\cite{barducci-bussi-parrinello-08} is a version of metadynamics with a vanishing penalization
mechanism where $\alpha(s)$ scales as a constant over $s$. We refer
to~\cite{dama-parrinello-voth-14,fort-jourdain-lelievre-stoltz-17,fort-jourdain-lelievre-stoltz-18}
for convergence analysis of well-tempered metadynamics using tools from stochastic approximation
algorithms. The
difficulty of approaches with vanishing adaption is to decide at which rate the penalization
strength should decrease: if it vanishes too fast, the system may
remain stuck in a metastable state; if it vanishes too slowly, large
fluctuations of the penalizing term affect the accuracy of the
estimator of the free energy, see for example~\cite{fort-jourdain-lelievre-stoltz-18} for discussion.

We are interested here  in the averaging approach. This technique has been used
in particular in the context of biased exchange metadynamics, see for
example~\cite{ghaemi-minozzi-carloni-laio-12,baftizadeh-biarnes-pietrucci-affinito-laio-12}. As
will be explained below (see also~\cite {laio-gervasio-08}), the
averaging approach is biased in nature: $\frac{1}{t} \int_0^t
\Psi_s(z) \, ds$ will not converge in the long time limit to $-F(z)$
(up to an additive constant),
except in the so-called {\em adiabatic} case, namely when 
$(\xi(Q_s))_{s \ge 0}$ follows an autonomous Markov
dynamics~\cite{iannuzzi-laio-parrinello-03,bussi-laio-parrinello-06}. The
adiabatic case is equivalent to a
situation in which the original dynamics is with values in ${\rm
  Ran}(\xi)$ and directly samples $\exp(-\beta
F(z)) \, dz$, so that one can consider simply the identity as a
collective variable. The fact that the averaging approach is biased is
clearly a drawback, but one interest compared to a vanishing adaption
method (such as well-tempered
metadynamics) is that one can look at the way the penalization term
$\Psi_t$ evolves in time
in order to assess if the underlying adiabatic assumption is satisfied:
if hidden metastabilities have not been taken into
account in the choice of $\xi$, one observes some kind of periodic behavior
of $\Psi_t$ in time. This will be illustrated below on numerical examples.

The objective of this work is twofold. First, we present in
Section~\ref{sec:adiabatic} some theoretical results on the unbiasedness of the averaging
approach in the case of adiabatic separation, for two versions of
metadynamics: a dynamics in continuous state space
where we built on previous results
from~\cite{benaim-ciotir-gauthier-15,benaim-gauthier-17} and a
dynamics in discrete state space --- the proof of the
results in this case are based on~\cite{TV11} and given in Section~\ref{sec:proof_discrete}. 
 Second, we give
some theoretical and numerical evidence of the fact that the average of the penalty term yields
a biased estimator of minus the free energy for non adiabatic cases in
Section~\ref{sec:non_adiabatic}.

\section{Consistency of metadynamics with adiabatic
  separation}\label{sec:adiabatic}

In this section, we prove the convergence of the average in time of
the penalty term to minus the free energy, in the adiabatic case, on
prototypical dynamics in continuous state space (Section~\ref{sec:adiabatic_cont}) and
in discrete state space (Section~\ref{sec:adiabatic_disc})

\subsection{A model in continuous state space}\label{sec:adiabatic_cont}

\subsubsection{Convergence results}
The results of this section can be seen
as a formalization of the pioneering
work~\cite{bussi-laio-parrinello-06}, where a similar model is considered.

Let us first introduce the (non-penalized) one-dimensional overdamped Langevin
dynamics which samples $Z^{-1} \exp(-\beta F(z)) \, dz$:
\begin{equation}\label{eq:OL}
dZ^{OL}_t = - F'(Z^{OL}_t) \, dt + \sqrt{2 \beta^{-1}} dB_t,
\end{equation}
where $\beta > 0$ is the inverse temperature and $(B_t)_{t \ge 0}$ is
a Brownian motion.
Compared to the setting presented in the introduction, we
are in the case of a so-called adiabatic separation between the
dynamics on $Z^{OL}_t=\xi(Q_t)$ and the dynamics of the other degrees of
freedom: the dynamics on $Z^{OL}_t$ is closed and Markovian, and thus, the
drift in~\eqref{eq:OL} is minus the derivative of the
free energy $F$, and the invariant measure is then $Z^{-1} \exp(-\beta
F(z)) \, dz$.
This dynamics will typically remain trapped in free energy wells. As
explained above, the idea of metadynamics is to penalize already
visited states in order to accelerate the sampling. The metadynamics
associated with~\eqref{eq:OL} (for a collective variable which is the
identity since we are in the adiabatic setting) gives the evolution of the one-dimensional
stochastic process $Z_t$ and an associated function $\Psi_t$, defined
formally as follows:
\begin{equation}\label{eq:cont_metadynamics}
\left\{
\begin{aligned}
dZ_t &= - \left( F'(Z_t) +  \Psi_t' (Z_t)\right) \, dt + \sqrt{2 \beta^{-1}}
dB_t, \\
\frac{d \Psi_t(z) }{dt} &= \gamma \delta^\epsilon(Z_t-z),
\end{aligned}
\right.
\end{equation}
where $\Psi_t'=\frac{d \Psi_t}{dz}$ denotes the derivative with
respect to $z$, $\gamma >0$ is the so-called {\em deposition rate} and
$\delta^\epsilon$ is a smooth even approximation of the Dirac
function. The initial conditions are $(Z_0, \Psi_0)$ where $Z_0$ is any
random variable independent of the Brownian motion $(B_t)_{t \ge 0}$
and $\Psi_0=0$. To give a proper meaning
to~\eqref{eq:cont_metadynamics} is not obvious since this is an
infinite dimensional stochastic differential equation~\cite{da_prato-zabczyk-14,benaim-ciotir-gauthier-15}. One aim
of this section is actually to exhibit a proper setting where this is well
defined, using~\cite{benaim-ciotir-gauthier-15}. For the moment, we argue formally.

In order to
simplify the presentation and build on previous results in the
literature~\cite{tarres-toth-valko-12,benaim-ciotir-gauthier-15,benaim-gauthier-17}, we assume that $Z_t$ lives on the one-dimensional
torus $\mathbb T$ with period $2\pi$, which amounts to assuming that both $F$ and $\delta^\epsilon$
are periodic functions with period $2 \pi$. The question we would like
to address is the following: Can we prove that $\lim_{t \to \infty}
\frac{1}{t} \int_0^t \Psi_s = -F +C$ where $C$ is an irrelevant constant to be
made precise?

Let us now introduce
$$\Phi_t(z)=\Psi_t(z)+F(z) - \frac{1}{2\pi}
\int_{\mathbb T} F.$$
We add the constant $- \frac{1}{2\pi} \int_{\mathbb T} F$ so that
$\Phi_0$ has zero mean, which will be convenient below to identify the
equilibrium measure of the dynamics (notice that only the $z$-derivatives
of $\Psi_t$ and $\Phi_t$ appear in the dynamics).
The dynamics~\eqref{eq:cont_metadynamics} can be rewritten as follows:
\begin{equation}\label{eq:cont_metadynamics_phi}
\left\{
\begin{aligned}
dZ_t &= - \Phi_t' (Z_t) \, dt + \sqrt{2 \beta^{-1}}
dB_t, \\
\frac{d \Phi_t(z) }{dt} &= \gamma \delta^\epsilon(Z_t-z),
\end{aligned}
\right.
\end{equation}
with initial conditions $Z_0$ and $\Phi_0=F-\frac{1}{2\pi}
\int_{\mathbb T} F$.

In order to go further, let us now make
$\delta^\epsilon$ precise. Starting from the relation in distribution
sense (for smooth $2\pi$ periodic test functions) $\delta =
\frac{1}{2\pi} + \frac{1}{\pi} \sum_{k \ge 1} \cos(k \cdot)$,
a possible choice is to consider the
truncated Fourier expansion
\begin{equation}\label{eq:choice1}
\delta^\epsilon(z)=\sum_{k=1}^{N} \cos(kz).
\end{equation}
We omit the additive constant $1/(2\pi)$: this is useful in order to
keep the normalization $\int_{\mathbb T} \Phi_t$ to zero as time evolves and thus to obtain
a stationary measure for $\Phi_t$ (otherwise $\int_{\mathbb T} \Phi_t$
would evolve linearly in time). Moreover, we forget the multiplicative constant $1/\pi$
since it can be taken into account by modifying $\gamma$. The dynamics~\eqref{eq:cont_metadynamics_phi}
then writes:
\begin{equation}\label{eq:cont_metadynamics_phi_prime}
\left\{
\begin{aligned}
dZ_t &= - \Phi_t' (Z_t) \, dt + \sqrt{2 \beta^{-1}}
dB_t, \\
\frac{d \Phi_t(z) }{dt} &= \gamma\sum_{k=1}^N \cos(kZ_t) \cos (kz)  +
\gamma \sum_{k=1}^N 
  \sin(kZ_t) \sin(kz).
\end{aligned}
\right.
\end{equation}
It is then natural to introduce the Fourier series associated with
$\Phi_t$ (remember that for all $t\ge 0$, $\int_{\mathbb T} \Phi_t(z) \, dz = 0$):
$$\Phi_t(z)=\sum_{k=1}^\infty \alpha_t^k \cos(kz) +\sum_{k=1}^\infty
\beta_t^k\sin(kz)$$
where $\alpha_t^k = \frac{1}{\pi} \int_{\mathbb T} \Phi_t(y) \cos(ky)
\, dy$ and $\beta_t^k = \frac{1}{\pi} \int_{\mathbb T} \Phi_t(y) \sin(ky)
\, dy$. In the following, we use the notation $\alpha_t=(\alpha^1_t, \ldots,
\alpha^N_t)$ and  $\beta_t=(\beta^1_t, \ldots,
\beta^N_t)$. Let us define 
\begin{align*}
\Pi_{N}^\perp(\Phi_0)(z)&= \Phi_0(z) -  \left( \sum_{k=1}^N \alpha_0^k \cos(kz) +\sum_{k=1}^N
\beta_0^k\sin(kz)\right)=\sum_{k=N+1}^\infty \alpha_0^k \cos(kz) +\sum_{k=N+1}^\infty
\beta_0^k\sin(kz),
\end{align*}
the projection of $\Phi_0= F(z)-\frac{1}{2\pi} \int_{\mathbb T} F$ over the Fourier modes
$(\cos(k\cdot),\sin(k\cdot))_{k \ge N+1}$. Notice that since $\Phi_0=F-\frac{1}{2\pi} \int_{\mathbb T} F$,
$$\Pi_{N}^\perp(\Phi_0)=\Pi_{N}^\perp (F).$$

The dynamics~\eqref{eq:cont_metadynamics_phi_prime} can now be rewritten
rigorously as the following
finite-dimensional stochastic differential equation:
\begin{equation}\label{eq:cont_metadynamics_alphabeta}
  \left\{
    \begin{aligned}
      dZ_t &= \sum_{k=1}^N k \left(\alpha^k_t \sin( k Z_t) -
        \beta^k_t \cos(k Z_t)\right) \, dt -
      \Pi_{N}^\perp (F)'(Z_t) \, dt + \sqrt{2 \beta^{-1}} dB_t, \\
      d\alpha^k_t &= \gamma \cos(kZ_t) dt \text{ for all } k \in \{1, \ldots,
      N\},\\
      d\beta^k_t &=  \gamma \sin(kZ_t) dt \text{ for all } k \in \{1, \ldots,
      N\},
    \end{aligned}
  \right.
\end{equation}
with initial condition $Z_0$, and $(\alpha^k_0,\beta^k_0)_{k \in \{1, \ldots,
      N\}}$. Equation~\eqref{eq:cont_metadynamics_alphabeta} gives
    a proper meaning to the original
    dynamics~\eqref{eq:cont_metadynamics} for the case $\delta^\epsilon$
    is defined by~\eqref{eq:choice1}.

Using techniques from~\cite{dolbeault-mouhot-schmeiser-15}, the following result is proven in~\cite{benaim-gauthier-17}:
\begin{prop}\label{prop:BG17}
Assume that
\begin{equation}\label{eq:hypo}
\Pi_{N}^\perp (F)=0.
\end{equation}
Then, the Markov process~\eqref{eq:cont_metadynamics_alphabeta} is a
positive Harris process and admits a unique invariant probability distribution
given as
$$\mu(dz,d\alpha,d\beta)=\frac{dz}{2\pi}  \prod_{k=1}^N
\frac{k^2}{2\pi \gamma} \exp\left( -\frac{k \left( \alpha_k^2 +\beta_k^2\right)}{2 \gamma}\right)
d \alpha_k \, d \beta_k.$$
Moreover, one has the following ergodicity results
\begin{itemize}
\item for all $f \in L^1(\mu)$, $\lim_{t \to \infty} \frac{1}{t}
  \int_0^t f(Z_s,\alpha_s,\beta_s)=\int_{\mathbb T \times \mathbb R^N \times
    \mathbb R^N} f d\mu$
\item there exists $\lambda > 0$ such that, for all initial condition $y_0=(z_0,\alpha_0,\beta_0) \in \mathbb T \times \mathbb R^N \times
    \mathbb R^N$ there exists $C(y_0)>0$ such that, for all $ t\ge 1$,
\[\|P_t(y_0,dy) - \mu(dy)\|_{TV}
    \le C(y_0) \exp(-\lambda t),\]
 where $P_t(y_0,dy)$ denotes the law of
    $(Z_t, \alpha_t, \beta_t)$ following the
    dynamics~\eqref{eq:cont_metadynamics_alphabeta} starting from the
    initial condition $(Z_0, \alpha_0, \beta_0)=y_0$.
\end{itemize}
\end{prop}

As a consequence, one obtains the following corollary.
\begin{cor}
Let us consider the dynamics~\eqref{eq:cont_metadynamics} with the
choice~\eqref{eq:choice1} for $\delta^\epsilon$. Under the
assumption~\eqref{eq:hypo}, one has
$$\lim_{t \to \infty} \frac{1}{t} \int_0^t \Psi_s ds = -F +
\frac{1}{2\pi} \int_{\mathbb T} F.$$
\end{cor}

It would be interesting to understand what can be said on the long-time
behavior of~\eqref{eq:cont_metadynamics_alphabeta} without the assumption~\eqref{eq:hypo}.

\subsubsection{Discussion and generalization}

\paragraph{On the diffusion constant.}
As explained in~\cite{benaim-gauthier-17}, the results of
Proposition~\ref{prop:BG17} hold whatever
the diffusion constant: $\sqrt{2\beta^{-1}}$ can be replaced by any positive
$\sigma$ in~\eqref{eq:cont_metadynamics}, and one still obtains the same limiting behavior. Moreover,
 for $\sigma=0$ (which actually corresponds to the original metadynamics
algorithm in~\cite{laio-parrinello-02}), the measure $\mu$ is still invariant for the
deterministic dynamical system, but it is not necessarily unique. We
refer to~\cite[Theorem 3]{benaim-gauthier-17} where the authors
identify infinitely many ergodic measures in the case $N=1$.

\paragraph{More general geometrical setting.}
We presented the results in Proposition~\ref{prop:BG17}  on the one-dimensional torus for simplicity,
but they actually hold for 
$Z_t$ with values in any compact Riemanian manifold, the functions
$\cos(kz)$ and $\sin(kz)$ being replaced by the
eigenfunctions of the Laplace Beltrami operator on the manifold (see~\cite{benaim-gauthier-17}).



\paragraph{On the infinite dimensional setting.}
An infinite-dimensional version
of the results of~\cite{benaim-gauthier-17} is given in~\cite{benaim-ciotir-gauthier-15}. In our context,
it can be translated as follows. Let us consider again the
dynamics~\eqref{eq:cont_metadynamics_phi}, with a function 
\begin{equation}\label{eq:choice2}
\delta^\epsilon(z)=\sum_{k=1}^\infty a_k \cos(kz)
\end{equation}
for some positive sequence $(a_k)_{k \ge 1}$ such that
$\sum_{k=1}^\infty (1+k^2)^5 a_k^2 < \infty$ (compare with~\eqref{eq:choice1}). For example, if
$\delta^\epsilon$ is a periodic function in the Sobolev space
$H^5(\mathbb T)$, then those properties are satisfied. As above, let us introduce the Fourier
series for~$\Phi_t$:
$$\Phi_t(z) = \sum_{k=1}^\infty \alpha_t^k \cos(kz) +
\sum_{k=1}^\infty \beta_t^k \sin(kz).$$
In this context, the dynamics~\eqref{eq:cont_metadynamics_alphabeta}
can be rewritten in the following form:
\begin{equation}\label{eq:cont_metadynamics_alphabeta_2}
  \left\{
    \begin{aligned}
      dZ_t &= \sum_{k=1}^\infty k \left(-\alpha^k_t (-\sin( k Z_t)) -
        \beta^k_t \cos(k Z_t)\right) \, dt + \sqrt{2 \beta^{-1}} dB_t, \\
      d\alpha^k_t &= \gamma a_k \cos(kZ_t) dt \text{ for all } k \ge 1,\\
      d\beta^k_t &=  \gamma a_k \sin(kZ_t) dt \text{ for all } k \ge 1.
    \end{aligned}
  \right.
\end{equation}
Let us assume
moreover that the function $F$ is such that:
\begin{equation}
\Phi_0(z)=F(z) - \frac{1}{2\pi} \int_{\mathbb T} F = \sum_{k = 1}^\infty
\alpha^k_0 \cos(kz) + \sum_{k = 1}^\infty
\beta^k_0 \sin(kz)
\end{equation}
for some sequences $(\alpha^k_0)_{k \ge 1}$ and $(\beta^k_0)_{k \ge
  1}$ such that
\begin{equation}\label{eq:hypo2}
(\alpha^k_0)_{k \ge 1} \in H \text{ and } (\beta^k_0)_{k \ge 1} \in H \text{ where } H=\left\{ (u_k)_{k \ge 1}, \,  \sum_{k = 1}^\infty \frac{u_k^2}{a_k} < \infty \right\}
\end{equation}
which can be seen as a generalization of the
condition~\eqref{eq:hypo}. The set of real-valued sequences $H$ is an
Hilbert space when endowed with the scalar product $\langle u, v
\rangle=\sum_{k = 1}^\infty \frac{u_k v_k}{a_k} $.
 Then, using an approach based on cylindrical processes and
 cylindrical measures
 associated with $H$, it can be shown that there exists a
unique strong
solution to~\eqref{eq:cont_metadynamics_alphabeta_2} such that $(\alpha^k_t)_{k \ge 1} \in
H$ and $(\beta^k_t)_{k \ge 1} \in H$ (see~\cite[Proposition
1]{benaim-ciotir-gauthier-15}), which admits as an invariant measure (see~\cite[Proposition
3]{benaim-ciotir-gauthier-15}):
$$\mu(dz,d\alpha,d\beta)=\frac{dz}{2\pi}  \prod_{k=1}^\infty
\frac{k^2}{2\pi \gamma} \exp\left( -\frac{k \left( \alpha_k^2 +\beta_k^2\right)}{2 \gamma}\right)
d \alpha_k \, d \beta_k.$$
However, in this infinite dimensional setting, it is unclear whether
the invariant measure is unique, and if the dynamics is ergodic with
respect to $\mu$ (see the discussion in~\cite[Section 7]{benaim-ciotir-gauthier-15}).

\paragraph{Enforcing the adiabatic separation.}
In the literature on metadynamics~\cite{iannuzzi-laio-parrinello-03},
some authors notice that the adiabatic separation may be
enforced by considering an extended system $(Q_t,Z_t)$ with potential energy
$V(q)+\frac{k}{2}(\xi(q) - z)^2$, and a Langevin dynamics where the
mass of the extended variable $Z_t$ is chosen sufficiently large,
see~\cite[Section 1.1]{laio-gervasio-08}. One then chooses
$\tilde{\xi}(q,z)=z$ as a reaction coordinate. Let us explain this
formally.

Let us consider the Langevin dynamics over the extended system, using
as a friction parameter on the additional variable $\gamma M^{1/2}$:
$$
\left\{
\begin{aligned}
dq_t&=p_t dt \\
dp_t&=-\nabla V(q_t) \, dt - k (\xi(q_t) - z_t) \nabla \xi(q_t) \, dt
- p_t \, dt + \sqrt{2 \beta^{-1}} dB_t \\
dz_t&= M^{-1} r_t dt\\
dr_t&= - k (z_t - \xi(q_t)) \, dt - \Psi_t'(z_t) \, dt - \gamma M^{-1/2} r_t \, dt +
\sqrt{2\gamma M^{1/2} \beta^{-1}} d\tilde{B}_t\\
\frac{d \Psi_t(z)}{dt}&=\gamma \delta^\epsilon(z_t -z).
\end{aligned}
\right.
$$
We proceed now in three steps. First, we use an averaging
principle~\cite{pavliotis-stuart-08}, by sending $M$ to $\infty$ after
changing variable $\tilde{r}_t=M^{-1/2}r_t$ and accelerating time by the
multiplicative factor $M^{1/2} $. One then obtains:
$$
\left\{
\begin{aligned}
dz_t&= \tilde{r}_t dt\\
d \tilde{r}_t&= - \tilde F'(z_t) \, dt - \Psi_t'(z_t) \,dt - \gamma
\tilde{r}_t \, dt + \sqrt{2\gamma \beta^{-1}} d\tilde{B}_t \\
\frac{d \Psi_t(z)}{dt}&=\gamma \delta^\epsilon(z_t -z)
\end{aligned}
\right.
$$
where
$$\tilde{F}(z)=- \beta^{-1} \ln \int \exp\left(-\beta V(q) - \beta
  \frac{k}{2}(\xi(q)-z)^2\right) \, dq.$$
Second, we consider the overdamped limit (see for
example~\cite[Section 2.2.4]{lelievre-rousset-stoltz-book-10}) by
sending $\gamma$ to $\infty$ after accelerating time by  the
multiplicative factor $\gamma$. One then obtains:
$$
\left\{
\begin{aligned}
d z_t&= - \tilde F'(z_t) \, dt - \Psi_t'(z_t) \,dt + \sqrt{2 \beta^{-1}} d\tilde{B}_t \\
\frac{d \Psi_t(z)}{dt}&=\gamma \delta^\epsilon(z_t -z).
\end{aligned}
\right.
$$
Finally, one easily checks that $\lim_{k \to \infty}
\tilde{F}(z)=F(z)$ (up to an irrelevant additive constant), where $F$ is the free energy defined
by~\eqref{eq:FE}, and one thus obtains formally the dynamics~\eqref{eq:cont_metadynamics}.


\subsection{A model in discrete space}\label{sec:adiabatic_disc}

We now introduce a
similar model by replacing the continuous circle by a discrete line.

\subsubsection{The model}\label{sec:the_model}
Let $K$ be a positive integer, and let $A:\{0,\ldots,K\}\to \dR$ be a
free energy 
profile. A good analogue of the overdamped Langevin process~\eqref{eq:OL}
in such a discrete setting is the continuous time, nearest neighbour Markov chain
on $\{0,\ldots,K\}$ with generator: for any test function $f:
\{0,\ldots,K\}\to \dR$, for any $k \in \{0,\ldots,K\}$,
\begin{equation}\label{eq:chaine_naive}
\begin{aligned}
  L^{OL}_{disc}f(k) =&  1_{\{k<K\}} \exp(-\beta (A_{k+1}-A_k)) (f(x,k+1)- f(x,k)) \\
        &   + 1_{\{k>0\}} \exp(-\beta (A_{k-1}-A_k)) (f(x,k-1) - f(x,k)). 
\end{aligned}
\end{equation}
It is easily seen that this Markov chain is reversible with respect to
the probability measure $Z^{-1} \exp(-2 \beta A_k)$ on $\{0,\ldots,K\}$ (and thus $2A_k$ plays the
role of the free energy $F$ in the continuous setting). 

To define a process that corresponds to the metadynamics algorithm, we introduce 
the local time at each site and we modify the jump dynamics so that the 
process is repelled by the sites where it has spent long times. 
More precisely, we define a process $(X_t,I_t)$ living in 
$\dR^K\times\{0, \ldots, K\}$.  The $I_t$ component describes the current position of a
walker on the discrete segment $\{0, \ldots, K\}$. The walker walks in a
potential given by the sum of a fixed landscape $A$ and a multiple of 
its own occupation
measure.  The continuous variable $X_t$ encodes, up to a multiplicative factor,
the \emph{differences} in occupation times for the process between neighbouring sites: denoting
by $L_t(k) = \int_0^t \ind{I_s = k} ds$ the local time at $t$ in $k$,
$X_t(k) = \gamma(L_t(k) - L_t(k-1))$. Similarly for $k\in\{1,\ldots,K\}$ let $$A'_k=
A_k - A_{k-1}.$$

The (adiabatic) metadynamics in discrete state space is the process defined by the following generator: for any test
function $f:\dR^K \times \{0,\ldots,K\} \to \dR$,
for all $ x=(x_1,\ldots,x_K)\in\dR^K$ and $k \in \{0,\ldots,K\}$, 
\begin{equation}\label{eq:generator_metadynamics}
  \begin{split}
    Lf(x,k) =& \gamma\PAR{1_{\{k>0\}}\partial_{x_k} f(x,k) - 1_{\{k<K\}} \partial_{x_{k+1}} f(x,k)} \\
  & + 1_{\{k<K\}} \exp(-\beta (x_{k+1}+A'_{k+1})) (f(x,k+1)- f(x,k)) \\
  & + 1_{\{k>0\}} \exp(\beta (x_k+A'_k)) (f(x,k-1) - f(x,k)).
				  \end{split}
\end{equation}
In plain words, while the walker $I_t$ stays in $k$, $X_t(k)$ increases at speed $\gamma$ and
$X_t(k+1)$ decreases at speed $\gamma$; the walker jumps to $k-1$ at rate
$\exp(\beta (x_k+A'_k))$ (except if it is at site $k=0$) and to $k+1$ at a rate $\exp(-\beta
(x_{k+1}+A'_{k+1}))$ (except if it is at site $k=K$). As in the continuous setting of
Section~\ref{sec:adiabatic_cont}, this corresponds to a dynamics in a
simplified setting (adiabatic dynamics) since there is no collective
variable involved here. We will discuss in Section~\ref{sec:non_adiabatic_disc} the
non-adiabatic version of this dynamics in discrete state space.

Note that, if the $X$ variable is frozen to the value $x=(x_1,\ldots,x_K)$,
the resulting jump process on $I$ is the naïve walk defined 
by~\eqref{eq:chaine_naive} in a potential $A_k+\gamma l_k$, where
$(l_k)$ satisfies $\gamma(l_k - l_{k-1}) = x_k$, which admits as an invariant measure $ Z_{x,A}^{-1} \exp( -
2\beta(A_k + \gamma l_k))$. 

We remark also that, for $t>0$, since $\forall k\in\{0,\hdots,K\}$, $\forall s\in[0,t]$,
$X_0(k)-\gamma t\le X_s(k)\le X_0(k)+\gamma t$, the jump rates are bounded on
the interval $[0,t]$ and there is no accumulation of jumps on this interval.
\begin{rem}[Related processes]
\label{rem:pdmp}
  This process interacts with its past \emph{via} its occupation measure which
  repels it.  Among the many variations on self-interacting processes (see the
  survey by Pemantle~\cite{Pem07}), it belongs to the family of so-called "true"
  or "myopic" self-avoiding walks. In particular, in  a flat potential ($A'_k=0$)
  and $\gamma=1$,  our model is exactly the same
  as the one studied  in  \cite{TV11}, except that we replace $\dZ$ with a finite
  set $\{0,\ldots,K\}$.   The main question addressed in \cite{TV11} is to prove
  scaling limits for the range of the walk and its position. We refer to the
  introduction of~\cite{TV11} for additional references on these models. 

 The process $(X_t,I_t)$ also fits the framework of  switching vector fields, a
 subclass of "Piecewise Deterministic Markov Processes" (PDMP) or  "hybrid
 processes" studied in particular in \cite{BH12,BLMZ12,BLMZ15} (see
 the reviews \cite{Mal14,ABK14} for more references on recent works on PDMPs). 
 Indeed, define the (constant) vector fields $(F_k)_{0\leq k\leq K}$ by
 $F_0 = (-\gamma,0,\ldots,0)$, $F_1 = (\gamma,-\gamma,0,\ldots,0)$, ..., 
 $F_{K-1} = (0,\ldots,0,\gamma,-\gamma)$, $F_K = (0,\ldots,0,\gamma)$.
 Then while $I_t = k$, $X_t$ follows the flow of $F_k$. 
 This will be used below to justify that the distribution of $X_t$ is nice
 enough. 
\end{rem}

\subsubsection{Longtime convergence results}

\begin{prop}[Invariant measure]\label{prop:mu}
  For $y\in\dR$ let $g_k(y)= \frac{2}{\beta\gamma} \cosh(\beta (y+A'_k))$. 
 Let  $m$ be the measure on $\dR^K$ defined by $dm(x) = dm(x_1,\ldots,x_K) =\exp\left(-\sum_{k=1}^Kg_k(x_k)\right) dx_1\cdots dx_K$. 
  The  probability  measure $\mu$ on $\dR^K \times \{0,\ldots,K\}$ defined by
\[
  \int f(x,k) d\mu(x,k) = \frac{1}{K+1} \sum_{0\leq k \leq K} \frac{1}{m(\dR^K)}  \int_{\dR^K} f(x,k) dm(x),
\]
is invariant for the Markov process
$(X_t,I_t)_{t \ge 0}$ with generator $L$ defined by~\eqref{eq:generator_metadynamics}. 
\end{prop}

\begin{rem}
The same model may be defined on the discrete circle, allowing for jumps between $K$ and $0$
in both directions. 
Letting $X_t(0) = L_t(0) - L_t(K)$, the computations are essentially the 
same as before, the invariant measure is given by a similar formula (note however that
it lives on the subspace $\sum X_t(k) = 0$). 
A key technical result used below (the Ray-Knight decomposition) is unfortunately
unavailable in this case. However, we believe our main results should still
hold by comparison arguments: informally speaking, the process does not see
that it lives in a circle until it has essentially filled all the potential
wells. 
\end{rem}

\begin{rem}[Link between the continuous model~\eqref{eq:cont_metadynamics} and the
  discrete model~\eqref{eq:generator_metadynamics}] 
To elucidate the link between~\eqref{eq:cont_metadynamics}
and~\eqref{eq:generator_metadynamics}, let us perform a change of
time on~\eqref{eq:cont_metadynamics} by changing $t$ to $\beta t$ so
that~\eqref{eq:cont_metadynamics} rewrites as follows (keeping the same notation
for the time-changed processes):
\begin{equation}\label{eq:cont_metadynamics_time_change}
\left\{
\begin{aligned}
dZ_t &= - \beta \left( F'(Z_t) +  \Psi_t' (Z_t)\right) \, dt + \sqrt{2}
dB_t, \\
\frac{d \Psi_t(z) }{dt} &= \beta \gamma \delta^\epsilon(Z_t-z).
\end{aligned}
\right.
\end{equation}
One can now identify the quantities involved in the discrete model
with the quantities involved in the continuous model as follows: $2A_k$ with $F(z)$, $I_t$ with $Z_t$, $2L_t$ with $\Psi_t$, $2X_t$
with $\Psi_t'$. To make the comparison complete, one would then need
to replace the deposition rate $\gamma_d$ for the discrete
model by $\gamma_c = \beta \gamma_d $. The reason why we used slightly different
notation for the discrete model than the standard notations for
metadynamics is because we would like to make the connection with the
fundamental paper~\cite{TV11} easier.
Notice that with this link drawn between the discrete and continuous
models, the invariant measures are  consistant: they are product
measures, with marginals the uniform law for $I_t$ and $Z_t$, and a
white noise for $X_t$ and $\psi_t'$. Moreover, one observes that the
potential function of the invariant measure for the $k$-th component
of $2 (X_t+A'_k)$ is then $-\frac{1}{\beta^2 \gamma} \cosh( \beta y /
2)$ which indeed compares well with the potential function of the
invariant measure for the Fourier components of $\Psi_t$ which are
$-y^2/ (2 \gamma k^{-2})$ in the regime $y \to 0$.
\end{rem}

The following invariance properties are easy to check. 
\begin{prop}[Identities in distribution]
  \label{thm:invariance}
  The processes for various values of the parameters are linked as follows: 
  \begin{enumerate}
  	\item 
  Let $(X_t,I_t)_{t \ge 0}$ start from $(X_0,I_0)=(0,i_0)$ in the landscape $A$
  (namely with generator~\eqref{eq:generator_metadynamics}). 
  Let $(Y_t,J_t)_{t \ge 0}$ start from $(Y_0,J_0)=(A',i_0)$ in a ``flat
  landscape'' (namely with generator~\eqref{eq:generator_metadynamics}
  with $A'=0$).  Then $(X_t+A',I_t)_{t \ge 0}$ and $(Y_t,J_t)_{t \ge 0}$ have the same distribution. 
\item\label{gamma1}
  Let $(X_t,I_t)_{t \ge 0}$ be the process in the landscape $A$ at inverse temperature
  $\beta$ and deposition rate $\gamma$. Then $(\gamma^{-1}X_t,I_t) _{t \ge 0}$ has the
  same distribution as the process in the landscape $\gamma^{-1}A$, at 
  inverse temperature $\beta\gamma$ and deposition rate $1$. 
  \end{enumerate}
\end{prop}


Proposition~\ref{thm:invariance} ensures that it is
enough to consider the case $A'_k=0$ and $\gamma=1$. In the following,
we thus consider the process $(X_t,I_t)_{t \ge 0}$, starting
from an initial distribution $(X_0,I_0) \sim \nu$
and with generator  defined by:
 for any
function $f:\dR^K \times \{0,\ldots,K\} \to \dR$,
for all $ x=(x_1,\ldots,x_K)\in\dR^K$ and $k \in \{0,\ldots,K\}$, 
\begin{equation}\label{eq:generator_metadynamics_flat}
  \begin{split}
    Lf(x,k) =& 1_{\{k>0\}}\partial_{x_k} f(x,k) - 1_{\{k<K\}} \partial_{x_{k+1}} f(x,k) \\
  & + 1_{\{k<K\}} \exp(-\beta x_{k+1}) (f(x,k+1)- f(x,k))  + 1_{\{k>0\}} \exp(\beta x_k) (f(x,k-1) - f(x,k)).
        			  \end{split}
\end{equation}
We denote by $\dP_\nu$ the law of this process, and write $\dP_{(x,i)} \egalDef \dP_{\delta_{(x,i)}}$ when the 
initial distribution is a Dirac mass. Consistently, one has
$g_k(y)=\frac{2}{\beta\gamma} \cosh(\beta y)$ in the definition of the
invariant measure $\mu$, see Proposition~\ref{prop:mu}.

The two main results are the exponential ergodicity of the process
$(X_t,I_t)_{t\ge 0}$ and a central limit theorem for the ergodic averages.
\begin{thm}\label{thmexperg}
  The process $(X_t,I_t)$ with generator~\eqref{eq:generator_metadynamics_flat} is exponentially ergodic~: there exists $\delta>0$, $C'_\delta<\infty$
 such that,  for any 
  starting point $(x,i)$, 
  \begin{equation}\label{eq:experg}
    \| \cL \left( (X_t,I_t) | (X_0,I_0) = (x,i)\right) - \mu \|_{TV} \leq C'_\delta W(x,i) e^{-\delta t}, 
  \end{equation}
  where $W(x,i) = \exp(\chi S(x))$ with 
\begin{equation}\label{eq:S}
S(x)=\sum_{i=0}^K \left(\left(\max_{0\le j\le K}
  \sum_{k=1}^jx_k\right) - \sum_{k=1}^ix_k\right)
\end{equation}
 (with the convention $\sum_{k=1}^0
  x_k=0$) and $\chi\in (0,+\infty)$ is a constant chosen in the proof of
  Theorem~\ref{thm:Vdecreases} below. 
\end{thm}
The function $S$ represents the ``computational sand'' needed to fill
the energy landscape defined by $l_k=\sum_{i=1}^k x_i$, see
Figure~\ref{fig:sand} below for a schematic representation.

In order to state precisely a result on ergodic averages, we need to 
ensure that they are defined. To this effect we introduce 
the following condition, linking a function 
$f$ and the initial measure $\nu$:
\begin{equation}
  \label{eq:finite_along_trajectories}
  \prb[\nu]{ \forall t\geq 0, \int_0^t \abs{f}(X_s,I_s) ds < \infty} = 1.
\end{equation}

\begin{thm}[Ergodic limit]\label{th:ergo}
  Let $f:\dR^K\times\{0,\hdots,K\} \to \dR$ be a measurable function such that $\int
  |f|d\mu<\infty$. If $f$ and the initial measure $\nu$ are such
  that \eqref{eq:finite_along_trajectories} holds, then 
 \[
 \frac{1}{t} \int_0^t f(X_s, I_s) ds \xrightarrow[t\to\infty]{\dP_\nu - \text{a.s.}} \int f d\mu.
  \]
\end{thm}

\begin{thm}[Central limit theorem]
  \label{thm:clt}
  Let $f:\dR^K\times\{0,\hdots,K\} \to \dR$ be a measurable function. 
 Assume that $f$ is in $L^p(\mu)$ for some $p>2$, and that $\int f d\mu = 0$. 
 Then 
  $\int_0^\infty|\esp[\mu]{f(X_0,I_0)f(X_t,I_t)}|dt<\infty$ and
  $c_f:=2\int_0^\infty\esp[\mu]{f(X_0,I_0)f(X_t,I_t)}dt\in[0,\infty)$.
  Moreover, if the initial measure $\nu$ is such that \eqref{eq:finite_along_trajectories}
  is satisfied, then 
  \[
 \frac{1}{\sqrt{t}} \int_0^t f(X_s, I_s) ds\mbox{ converges in law to }\cN(0, c_f)\mbox { under }\dP_{\nu}\mbox{ as }t\to\infty.
  \]
  \end{thm}

  \begin{cor}[Learning the free energy differences]\label{cor:TCL}
    Consider the process in a landscape $A$, that is, with generator given by
    \eqref{eq:generator_metadynamics}. For any $k$, there exists $c_k\in[0,\infty)$ such that
    for any initial measure $\nu$, the time average $M_t(k) \defEgal  \frac{1}{t} \int_0^t X_s(k) ds$
    satisfies:
    \begin{align*}
      M_t(k) &\xrightarrow[t\to\infty]{\text{a.s.}}  - A'_k,   &
                                                                 \sqrt{t}
                                                                 \left(
                                                                 M_t(k)
      +A_k'\right)&\xrightarrow[t\to\infty]{(d)} \cN(0,c_k).
    \end{align*}
  \end{cor}
  Under the invariant measure $\mu$, the $X(k)$'s are
  independent. Moreover, in the landscape $A$,
 the marginal density of $X(k)$ is symmetric  around its mean value $-A'_k$. 
Corollary~\ref{cor:TCL} implies that the 
  ergodic means $\frac{1}{t} \int_0^t X_s(k) ds$ converge to their 
  expected value $-A'_k$. In this sense, the process $X_t$  "learns"
  the derivative of the free energy profile $A$.

\begin{rem}[On the finitess condition \eqref{eq:finite_along_trajectories}]
  There are many couples $(f,\nu)$ for which  \eqref{eq:finite_along_trajectories}
  is easily checked. 
   \begin{itemize}
\item Since the process moves at a finite speed, in the sense that 
  \[ \prb[(x,i)]{\forall k\in\{0,\hdots,K\},\,\forall s\in[0,t],\,|X_s(k)|\le |x_k|+t}=1,\]
  the condition~\eqref{eq:finite_along_trajectories} is satisfied for all initial measures
  $\nu$ if $f$ is locally bounded on $\dR^K\times\{0,\hdots,K\}$.
\item Let $f\in L^1(\mu)$. Since $\mu$ is invariant, 
  \eqref{eq:finite_along_trajectories} is satified for $\nu = \mu$. 
  Consequently, there is a measurable set $S_f$ of starting points 
  such that $\mu(S_f) = 1$ and~\eqref{eq:finite_along_trajectories} is
  satisfied for $\nu = \delta_{(x,i)}$ for all $(x,i)\in S_f$.
  Consequently  \eqref{eq:finite_along_trajectories} also 
  holds if $\nu$ is absolutely continuous with respect to $\mu$. 
\item When $K=1$, the condition~\eqref{eq:finite_along_trajectories} is satisfied 
  for any initial measure $\nu$ as soon as $f\in L^1(\mu)$. Indeed, denoting by $N_t$ the a.s. finite number of jumps of the
  component $I_s$ on $[0,t]$, one checks by performing a change of variables
  between $0$ and the first jump time then between two successive of these jump
  times, last between the last one and $t$ that for each measurable function
  $\varphi:\dR\to\dR_+$ 
  \[\int_0^t\varphi(X_s)ds\le 
(N_t+1)\int_{\inf_{s\in[0,t]}X_s}^{\sup_{s\in[0,t]}X_s}\varphi(x)dx\le 
(N_t+1)\exp\left(g_1\left(\inf_{s\in[0,t]}X_s\right)\vee g_1\left(\sup_{s\in[0,t]}X_s\right)\right)\int\varphi(x)dm(x), \]
so that the left hand side is a.s.\ finite when $\varphi\in L^1(m)$. 
\item
  By contrast, if $K\geq 2$ then one may construct couples $(f,\nu)$ for which 
  $f\in L^1(\mu)$ but \eqref{eq:finite_along_trajectories} is \emph{not} satisfied. 
  Indeed, for the initial measure $\nu = \delta_{((0,...,0),K)}$, there is a positive
  probability to observe the trajectory $(X_s,I_s) = ((0,...,0,s),K)$ until time $1$, 
  and one can easily define an $f\in L^1(\mu)$ such that the integral $\int_0^1 f((0,...,0,s),K) ds$ 
  diverges. 
\end{itemize}
\end{rem}

The main ingredient to prove these results is the following
control on return times to a compact set. 
\begin{thm}
  \label{thm:expMoments_continuous}
There exists $\eta>0$, a compact set
$\cK\subset \dR^K$ (defined in Section \ref{sec:expMoments} below), and  
two  constants $\tilde{C}>0$ and $\delta>0$ such that for all $(x,i) \in \dR^K \times \{0, \ldots,K\}$,
  \[ \esp[(x,i)]{\exp\left(\delta \tau_\cK(\eta)\right)} \leq
    \tilde{C} \exp(\chi S(x)) <\infty,\]
  where $S(x)$ has been defined in~\eqref{eq:S} and $\tau_\cK(\eta) = \inf\{ t\geq \eta,
  X_t\in\cK\}$ is the first return time to $\cK$ after the time $\eta$
  for the process $(X_t,I_t)_{t \ge 0}$ with generator~\eqref{eq:generator_metadynamics_flat}.
\end{thm}

\subsubsection{A key tool: the Ray-Knight representation}
\label{sec:RayKnight}

We consider here the process $(X_t,I_t)_{t \ge 0}$ with $(X_t(1),\ldots,X_t(K))\in\dR^K$ 
and $I_t\in\{0,\ldots,K\}$ with generator~\eqref{eq:generator_metadynamics_flat}
(flat landscape and $\gamma=1$), starting from an arbitrary point $(x_0,i_0)$. 
Recall that $L_t(k) = \int_0^t \ind{I_s = k} ds$ is the local time --- namely the
time spent by the discrete coordinate $I_s$ on site $k$ before time $t$. 
Note that with this notation,
\[
\text{ for } t \ge 0 \text{ and } k \in \{1, \ldots, K\}, \,  X_t(k) = x_0(k) + L_t(k) - L_t(k-1).
\]
The key tool in the proof of the control on the return times
(Theorem~\ref{thm:expMoments_continuous}) is a very nice representation of the local
time profile $k \mapsto L_t(k)$ as a random walk, in the spirit of the classical Ray-Knight
theorem for the Brownian motion (see for example \cite[Chapter XI.2]{RY99}). 
This representation is established in~\cite{TV11} in a very similar setting. 
We describe it here,  using very similar notation to~\cite{TV11}.

For $r\geq 0$ and $j\in\{0,\ldots,K\}$, let $T_{j,r}$ be the first time when $L_t(j)$ exceeds $r$: 
\[ 
  T_{j,r} = \inf\{ s\geq 0, L_t(j) \geq r\}. 
\]
Let $\Lambda_{j,r}(k)$ be the local time profile at the "inverse local
time" $T_{j,r}$: for all $k \in \{0, \ldots K\}$,
\[ \Lambda_{j,r}(k) = L_{T_{j,r}}(k). \]
Thus, $\Lambda_{j,r}(k)$ is the time spent in $k$ when $L_t(j)$
reaches $r$.

Ray-Knight theorems describe the law of the local time profile $k
\mapsto \Lambda_{j,r}(k)$: 
in our case, as shown in \cite{TV11}, it is given by a random walk
with ``time'' $k$, and an explicit distribution of the increments
$\Lambda_{j,r}(k+1)-\Lambda_{j,r}(k)$, as a function of $k$ and
$\Lambda_{j,r}(k)$. To specify this distribution, let us introduce three (families of) auxiliary processes.

For $k\in\{1,\ldots,K\}$, let  $\xi_k(s)$ be the $k$-th coordinate $X(k)$ "viewed only when it
changes". The construction is illustrated in Figure~\ref{fig:defXiEta}. 
More formally, let $\tau_t(k) = L_t(k) + L_t(k-1)$ record the time
spent in $\{k-1,k\}$, let $\theta_s(k) = \inf\{ t, \tau_t(k) \geq s\}$ be its inverse, 
and let
\[ 
  \xi_k(s) = X_{\theta_s(k)} (k).
\]
Let $\alpha_k(s) \in \{-1,1\}$ be the "discrete velocity"
\[
  \alpha_k(s) = \ind{I_{\theta_s(k)} = k} - \ind{I_{\theta_s(k)} = k-1}. 
\]
At the time $s$ for the "local" clock, either $\alpha_k = 1$ and $\xi_k$ is moving up, 
or $\alpha_k = -1$ and $\xi_k$ is moving down. 

\begin{figure}
  \centering
\begin{tikzpicture}[
  flat/.style={dashed},
  up/.style={very thick,red},
  down/.style={very thick,blue},
  scale=0.7
  ]
  \draw[help lines,->] (0,0)--(17,0);
  \draw[help lines,->] (0,-2)--(0,3);
  \path (0,0)  coordinate (a) 
      ++(2,0)  coordinate (b) 
      ++(2,2)  coordinate (c) 
      ++(3,-3) coordinate (d) 
      ++(1,0)  coordinate (e) 
      ++(1,-1) coordinate (f)
      ++(2,2)  coordinate (g)
      ++(3,0)  coordinate (h)
      ++(1,1)  coordinate (i)
      ++(2,-2)  coordinate (j);
   \draw[up] (b)--(c) (f)--(g) (h)--(i) ;
   \draw[down] (c)--(d) (e)--(f) (i)--(j);
   \draw[flat] (a)--(b) (d)--(e) (g)--(h);
\end{tikzpicture}

\medskip

\begin{tikzpicture}[
  flat/.style={dashed},
  up/.style={very thick,red},
  down/.style={very thick,blue},
  scale=0.7
  ]
  \draw[help lines,->] (0,0)--(17,0);
  \draw[help lines,->] (0,-2)--(0,3);
  \path (0,0)  coordinate (a) 
      coordinate (b) 
      ++(2,2)  coordinate (c) 
      ++(3,-3) coordinate (d) 
      coordinate (e) 
      ++(1,-1) coordinate (f)
      ++(2,2)  coordinate (g)
      coordinate (h)
      ++(1,1)  coordinate (i)
      ++(2,-2)  coordinate (j);
   \draw[up] (b)--(c) (f)--(g) (h)--(i) ;
   \draw[down] (c)--(d) (e)--(f) (i)--(j);
\end{tikzpicture}

\medskip

\begin{tikzpicture}[
  flat/.style={dashed},
  up/.style={very thick,red},
  down/.style={very thick,blue},
  scale=0.7
  ]
  \draw[help lines,->] (0,0)--(17,0);
  \draw[help lines,->] (0,-2)--(0,3);
  \path (0,0)  coordinate (a) 
      ++(0,2)
      coordinate (b) 
      coordinate (c) 
      ++(3,-3)
      coordinate (d) 
      coordinate (e) 
      ++(1,-1)
      coordinate (f)
      ++(0,3)
      coordinate (g)
      coordinate (h)
      coordinate (i)
      ++(2,-2)  coordinate (j);
   \draw[down] (c)--(d) (e)--(f) (i)--(j);
   \draw[up,dotted]   (a)--(b) (f)--(g);
\end{tikzpicture}
\caption{Top figure: trajectory of $X_t(k)=L_t(k) - L_t(k-1)$. Middle figure: corresponding
  trajectory $\xi_k(s)$ --- the flat parts (when $I_t \notin \{k-1,k\}$) have been removed. 
  Bottom figure: corresponding trajectory $\eta_k^-(s)$ ---  the upward parts (when
  $I_t = k$) have been removed. }\label{fig:defXiEta}
\end{figure}
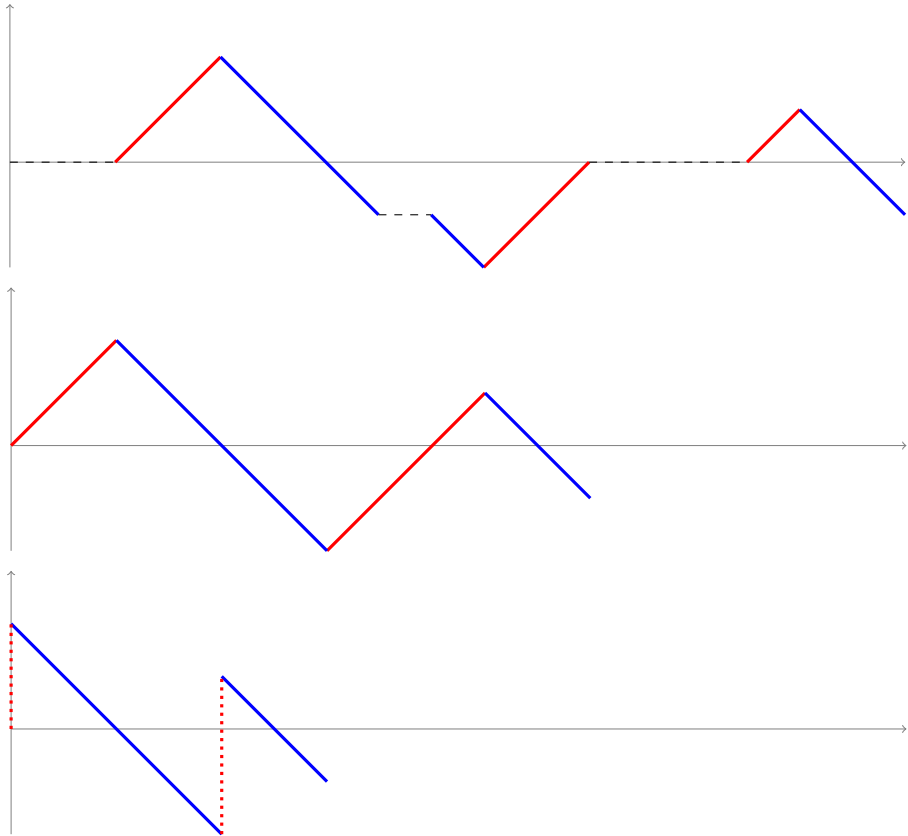


The following facts are easy to check: 
\begin{enumerate}
	\item The $(\xi_k,\alpha_k)$ are one-dimensional PDMPs with a
          common generator (whatever $k \in \{1, \ldots, K \}$)
	  \[
 \forall x \in \dR^K, \,	    Gf(x,\pm 1) = \pm f'(x,\pm 1) + \exp(\pm \beta x) (f(x, \mp 1) - f(x,\pm 1)).
	  \]
	\item The continuous component $\xi_k$ starts from $x_0(k)$. The discrete
	  component $\alpha_k$ starts from $-1$ if $k> i_0$ and from $1$ if $k\leq i_0$. 
	\item The $(\xi_k,\alpha_k)$ for different values of $k$ are independent. 
\end{enumerate}
\begin{rem} The two minor differences with the corresponding statement
  in \cite{TV11} (Proposition 2.1) is that $\xi_k$ does not start from $0$
  and there is a shift in indices, due to differing conventions: our
  $X_k$ is defined on the edge $\{k-1,k\}$ rather than $\{k,k+1\}$. 
\end{rem}

The second family of processes $\eta^-_k$ is defined similarly by looking at $X_t(k)$
only when $I_t$ is in $\{k-1\}$ (see Figure~\ref{fig:defXiEta}) and
the third family $\eta^+_k$ is defined by looking at $(-X_t(k))$ when
$I_t$ is in $k$:
$$\eta^\pm_k(s)= \mp X_{\theta^\pm_s(k)}(k) \text{ where } \theta^\pm_s(k)=
\inf \{t, \tau^\pm_t(k) \ge s\} \text{ with } \tau^+(k)=L_t(k) \text{
    and } \tau^-(k)=L_t(k-1).
$$

Still quoting \cite[Proposition 2.1]{TV11}, we get
\begin{thm}\label{th:ray-knight}
  The PDMP's $\eta^{\pm}_k$ for different values of $k$ are independent. 
  Their common  evolution is given by the generator
  \begin{equation}\label{eq:H}
 Hf(y) = -f'(y) + \exp(-\beta y) \int_y^\infty q(y,z) (f(z) - f(y))
 dz,
\end{equation}
  where the normalized jump kernel $q$ is given by
  \begin{equation}\label{eq:q}
    q(y,z) = \ind{z\geq y} e^{\beta z} \exp\PAR{ - \int_y^z e^{\beta s} ds}.
  \end{equation}
  Their initial distributions are given by: 
  \[
  \begin{array}{r|c|c}
         &\eta^-_k & \eta^+_k \\
	 \hline
 i_0< k  &\delta_{x_0(k)} & q(-x_0(k),z) \, dz \\
 \hline
 i_0\geq k  &q(x_0(k),z) \, dz  & \delta_{-x_0(k)} 
 \end{array}
 \]
\end{thm}
\begin{rem}
The derivative in the generator of $\eta^-_k$ comes from the fact that
$X_t(k)=L_t(k)-L_t(k-1)$ decreases with velocity $-1$ when $I_t=k-1$. When
$I_t$ moves to smaller values, it must come back to $k-1$ before reaching~$k$
so that no jump is observed. When~$I_t$ moves to higher values, only the time
spent at $k$ before returning to $k-1$ affects $\eta^-_k$: the excursions to
the right of $k$ have no influence. When $X(k)=y$ and $I=k-1$, $I$ moves to $k$
with rate $e^{-\beta y}$ and since the jump rate from $k$ to $k-1$ is given by
$e^{\beta x_k}$ and $X(k)$ increases with velocity $1$ as long as $I$ is equal
to $k$, the probability that the total time spent by $I$ in $k$ before it
returns to $k-1$ is larger than $t$ is $e^{-\int_0^te^{\beta(y+s)}ds}$ so the
density of the value of $X(k)$ when $I$ returns to $k-1$ is $q(y,z)=1_{\{z\ge
  y\}}e^{\beta z}e^{-\int_y^ze^{\beta s}ds}$. Once more, the only differences with \cite{TV11} are a shift in indices and
different initial values.
\end{rem}
 One has
\begin{equation}\label{eq:1}
\Lambda_{j,r}(k)-\Lambda_{j,r}(k-1)=L_{T_{j,r}}(k)-L_{T_{j,r}}(k-1)=X_{T_{j,r}}(k)-x_0(k).
\end{equation}
Suppose now that $k>j$. Then, to reach $j$ from $k$, $I_t$ spends some time at
$k-1$ after the last visit to $k$ so that
\begin{equation}\label{eq:2}
X_{T_{j,r}}(k)=\eta_k^-(\Lambda_{j,r}(k-1)).
\end{equation}
Equation~\eqref{eq:2} is easy to check by noting that by definition of
$\eta_k^-$  and $\Lambda_{j,r}$,
$\eta_k^-(\Lambda_{j,r}(k-1))=X_{\theta^-_{L_{T_{j,r}}(k-1)}(k)}(k)$
and then, by definition of $\theta^-$, $\theta^-_{L_s(k-1)}(k)=s$.

 From~\eqref{eq:1}-\eqref{eq:2}, one gets:
\begin{equation}
  \label{eq:ray-knight-lambda}
  \Lambda_{j,r}(k) = \Lambda_{j,r}(k-1) + \eta^-_k(\Lambda_{j,r}(k-1)) - x_0(k).
\end{equation}

Therefore the local time profile on the right of $j$ is given by a walk: the 
distribution of the difference $\Lambda_{j,r}(k) - \Lambda_{j,r}(k-1)$ only
depends on the "current position" $\Lambda_{j,r}(k-1)$ and the "time" $k$. 
The size of the step is obtained by running the $\eta^-_k$ process for 
a time $\Lambda_{j,r}(k-1)$ independent of this process (since $\Lambda_{j,r}(j)=r$ and the $\eta_k^-$ are independent for different values of $k$).

Similarly the walk "on the left" is given for $k\leq j$ by
\[
  \Lambda_{j,r}(k-1) = \Lambda_{j,r}(k) + \eta^+_k(\Lambda_{j,r}(k)) + x_0(k).\]

This very precise description of the local time profile will be 
essential in our proof of Theorem~\ref{thm:Vdecreases} below. 

\subsection{On the efficiency of metadynamics in the adiabatic case}

In the adiabatic case, the results presented in Section~\ref{sec:adiabatic_cont} for the
continuous model and in Section~\ref{sec:adiabatic_disc} for the discrete model show
that the time average of the penalty term ($\Psi_t$ for the continuous
model and $L_t$ for the discrete model) converges to minus the free energy
profile in which the original non-penalized dynamics evolves (namely~\eqref{eq:OL}
and~\eqref{eq:chaine_naive}). We would like to discuss here the
efficiency of the approach to approximate the free energy, both at
equilibrium and in the transient regime, compared to an estimator
based on the non-penalized dynamics (say $-\frac{1}{\beta} \ln \int_0^t
\delta^\epsilon(Z^{OL}_t -z)$ for the continuous dynamics~\eqref{eq:OL}, and  $-
\frac{1}{2\beta} \ln \int_0^t \ind{I^{OL}_s = k} ds$ for the discrete
dynamics, $(I^{OL}_t)_{t \ge 0}$ being the
Markov process with generator $L^{OL}_{disc}$, see\eqref{eq:chaine_naive}).

First, we observed in both cases that at equilibrium, the asymptotic
variance of the estimator of the free energy derivative does not
depend on the free energy profile $F$ and $A$. This is a consequence of the
``trick'' which makes the free energy
 disappear by changing the initial condition on $\Psi_t$ and $X_t$
 (this trick was already
 mentioned in~\cite{bussi-laio-parrinello-06}). This means that even if $F$ or $A$ have
 wells separated by very large barriers, these are not seen in the
 longtime limit (which is not the case for estimators based on the non-penalized dynamics).

On the other hand, the free energy profile does influence the
transient regime (the ``burn-in'' time). Let us focus here on the
discrete model. By Proposition~\ref{thm:invariance}, the initial free energy profiles
appears in the initial condition of the process we study and the explicit bounds for the Lyapunov function~$W$
ensure that, starting from  $X_0 = A'$, the time $T$ needed to reach a compact set near $0$ 
is at most linear in $S(A')$, which is the ``computational sand'' needed
to fill the wells of $A$ (see Figure~\ref{fig:sand}): 
\[
  \prb[X_0=A']{T>t} \leq \tilde{C} \exp(-\delta t) \exp(\chi S(A')),
\]
see Theorem~\ref{thm:expMoments_continuous}.

Let us finally discuss the role of $\gamma$. One expects that, in the
transient regime, the
larger  $\gamma$, the quicker the filling of the wells. On the
other hand, the asymptotic
 variance increases as $\gamma$ increases. There is thus a balance to
 find on the deposition rate, to optimize the efficiency both in the
 transient and the asymptotic regimes (see the discussion
 after~\cite[Theorem 3.6]{fort-jourdain-kuhn-lelievre-stoltz-15}
 and~\cite[Section 5.1]{fort-jourdain-lelievre-stoltz-17}  for similar
 considerations on the stepsize sequence for algorithms with vanishing
 adaption, namely with a deposition rate such that $\gamma(t) \to 0$ when $t \to \infty$).

\section{Non consistency of metadynamics in the absence of adiabatic separation}\label{sec:non_adiabatic}

The objective of this section is to show that, for both the continuous
and the discrete metadynamics, the average of the penalization yields
a biased estimator of the free energy in the absence of adiabatic
separation. We also illustrate the fact that this bias goes to zero
when the deposition rate $\gamma$ goes to zero. 

\subsection{Continuous state space metadynamics: a numerical study}

We consider a two-dimensional potential $V:\dT\times \dR \to \dR$ where
$\dT=\dR / (2 \pi \dZ)$ denotes the interval $(-\pi,\pi]$ with periodic
boundary conditions.  The dynamics is the following:
\begin{equation}\label{eq:MD}
\left\{
\begin{aligned}
dx_t =& - \left(\partial_x V(x_t,y_t) + \frac{d\Psi_t}{dx}
  (x_t)\right) \, dt + \sqrt{2\beta^{-1}} dB^x_t \\
dy_t=& -\partial_y V(x_t,y_t) \, dt + \sqrt{2\beta^{-1}} dB^y_t\\
d\Psi_t(x) =& \frac{\gamma}{\sqrt{2 \pi} \varepsilon}
\exp\left(-\frac{(x-x_t)^2}{2 \varepsilon^2}\right)  \, dt
\end{aligned}
\right.
\end{equation}
where $\gamma > 0$ is the deposition rate and $\varepsilon >0$ is
the width of the Gaussian. Notice that $\Psi_t(x)$ is defined for $t
\ge 0$ and $x \in \dT$: the periodic images of the Gaussians are
taken into account.

Let us emphasize that we use periodic boundary conditions in the
$x$-direction in order to avoid difficulties related to the correct
implementation of boundary conditions for metadynamics, see~\cite{bussi-laio-parrinello-06,CMPL10}.

The precise form of the potential is (see Figure~\ref{fig:V})
$$V(x,y)=\cos(2x)+0.05 (y-3 \cos(2x)-3)^2 + 0.5 \sin(x).$$
The inverse temperature is $\beta^{-1}=50$.
In this setting, the free energy is analytically computable and
defined, up to an additive constant, by:
$$\forall x \in \dT, \, F(x) = - \beta^{-1} \ln \int_{\dR} \exp(-\beta
V(x,y)) \, dy = \cos(2x) + 0.5\sin(x)$$
so that its derivative is
$$\forall x \in \dT, \, F'(x)=-2 \sin(2x) + 0.5 \cos(x).$$

\begin{figure}
\centerline{\includegraphics[width=12cm]{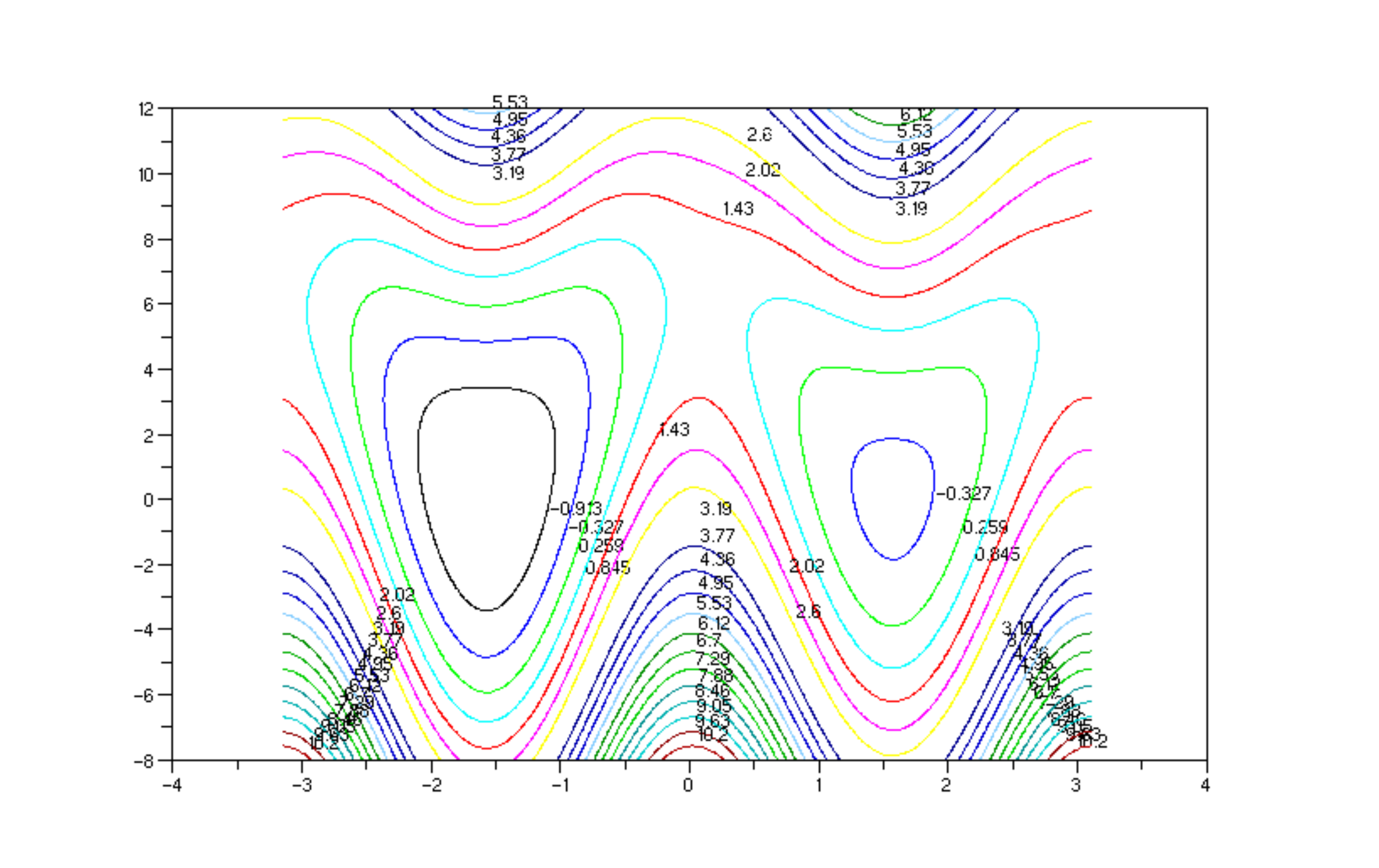}}
\caption{Contour plot of the potential $V$.}\label{fig:V}
\end{figure}

The dynamics is discretized using an Euler-Maruyama scheme with
timestep $\delta t=10^{-4}$ over the time interval $[0,10^4]$.
 We use a mesh of $(-\pi,\pi]$ with $I=40$ intervals of the same size.
 The function $x \mapsto \Psi_t(x)$ is approximated over this mesh
 using continuous piecewise affine interpolation. The width of the Gaussians is set to the size of the intervals: $\varepsilon=2\pi/I$.

In the following, we compare,  for
$x \in \dT$,  $-F'(x)$ with $\lim_{t \to \infty}
\frac{1}{t} \int_0^t \frac{d\Psi_s}{dx}(x) \, ds$. We have checked numerically
that the longtime limit is attained at $t=10^4$.

We observe on Figures~\ref{fig:1},~\ref{fig:2} and~\ref{fig:3} that the
quality of the result highly depends on the choice of the parameter
$\gamma$: as $\gamma$ decreases, the result gets closer to the
expected answer. This justifies the use of a parameter $\gamma$
depending on time $t$ and going to zero in the longtime limit
(vanishing adaption), as in
Wang Landau~\cite{wang-landau-01-PRL}, in Well Tempered Metadynamics~\cite{barducci-bussi-parrinello-08} or in Self Healing Umbrella Sampling~\cite{marsili-barducci-chelli-procacci-schettino-06}.

\begin{figure}
\centerline{\includegraphics[angle=270,width=12cm]{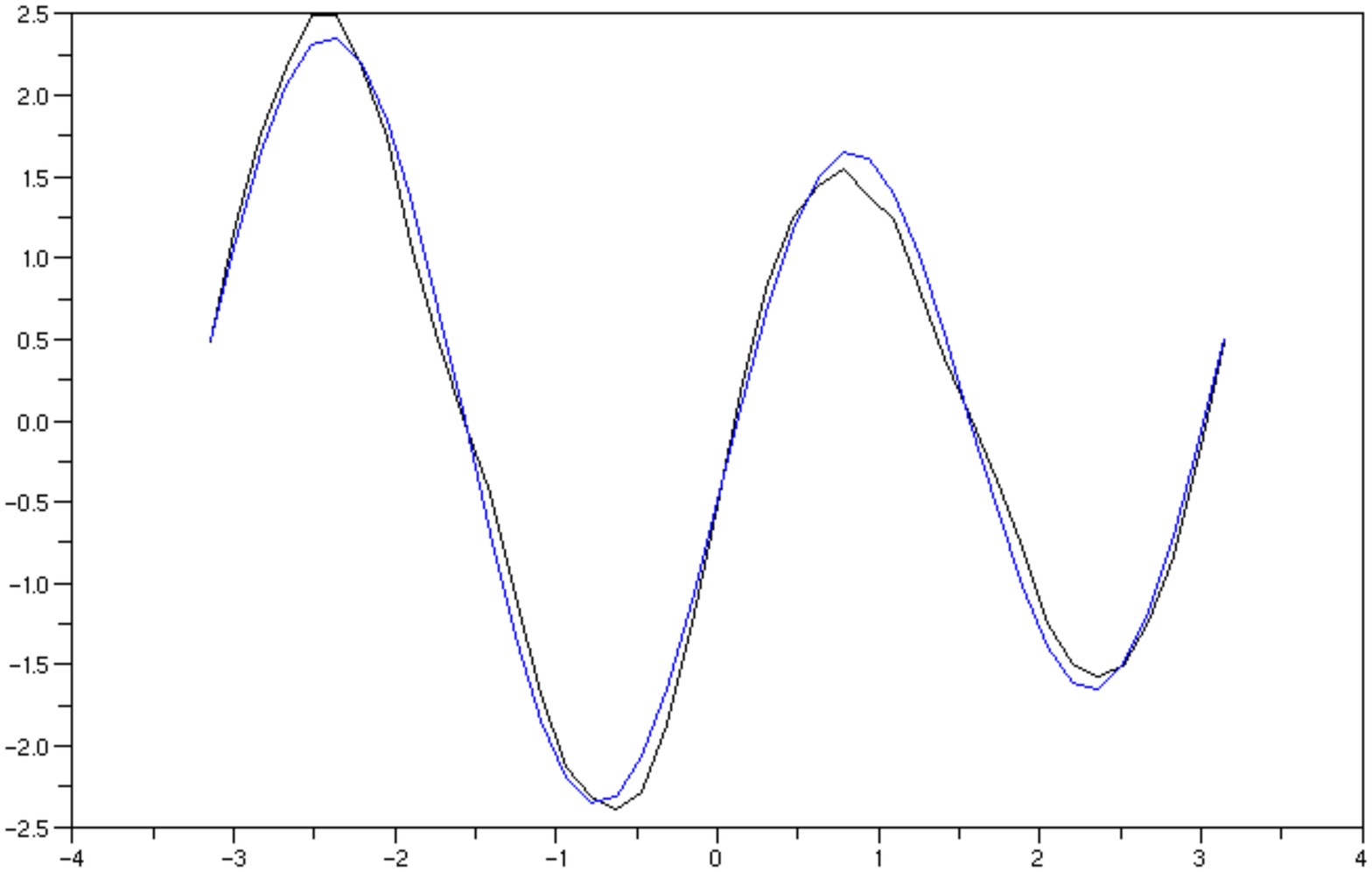}}
\caption{Comparison of the mean force profiles,  $\gamma=0.1$.  Blue:
  $-F'(x)$. Black: $\lim_{t \to \infty}
\frac{1}{t} \int_0^t \frac{d\Psi_s}{dx}(x) \, ds$.}  \label{fig:1}%
\end{figure}

\begin{figure}
\centerline{\includegraphics[angle=270,width=12cm]{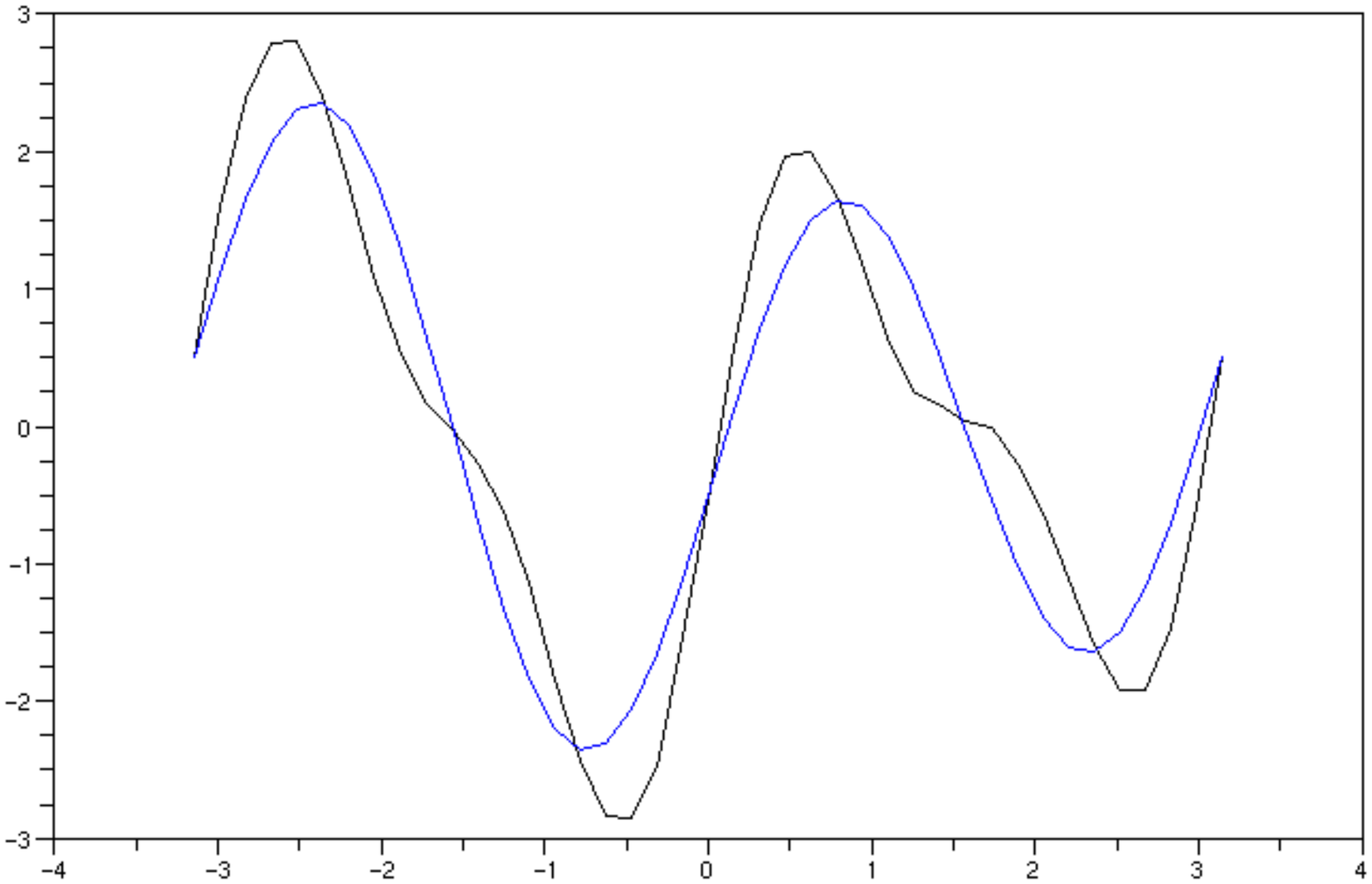}}
\caption{Comparison of the mean force profiles,  $\gamma=1$. Blue:
  $-F'(x)$. Black: $\lim_{t \to \infty}
\frac{1}{t} \int_0^t \frac{d\Psi_s}{dx}(x) \, ds$.}  \label{fig:2}%
\end{figure}

\begin{figure}
\centerline{\includegraphics[angle=270,width=12cm]{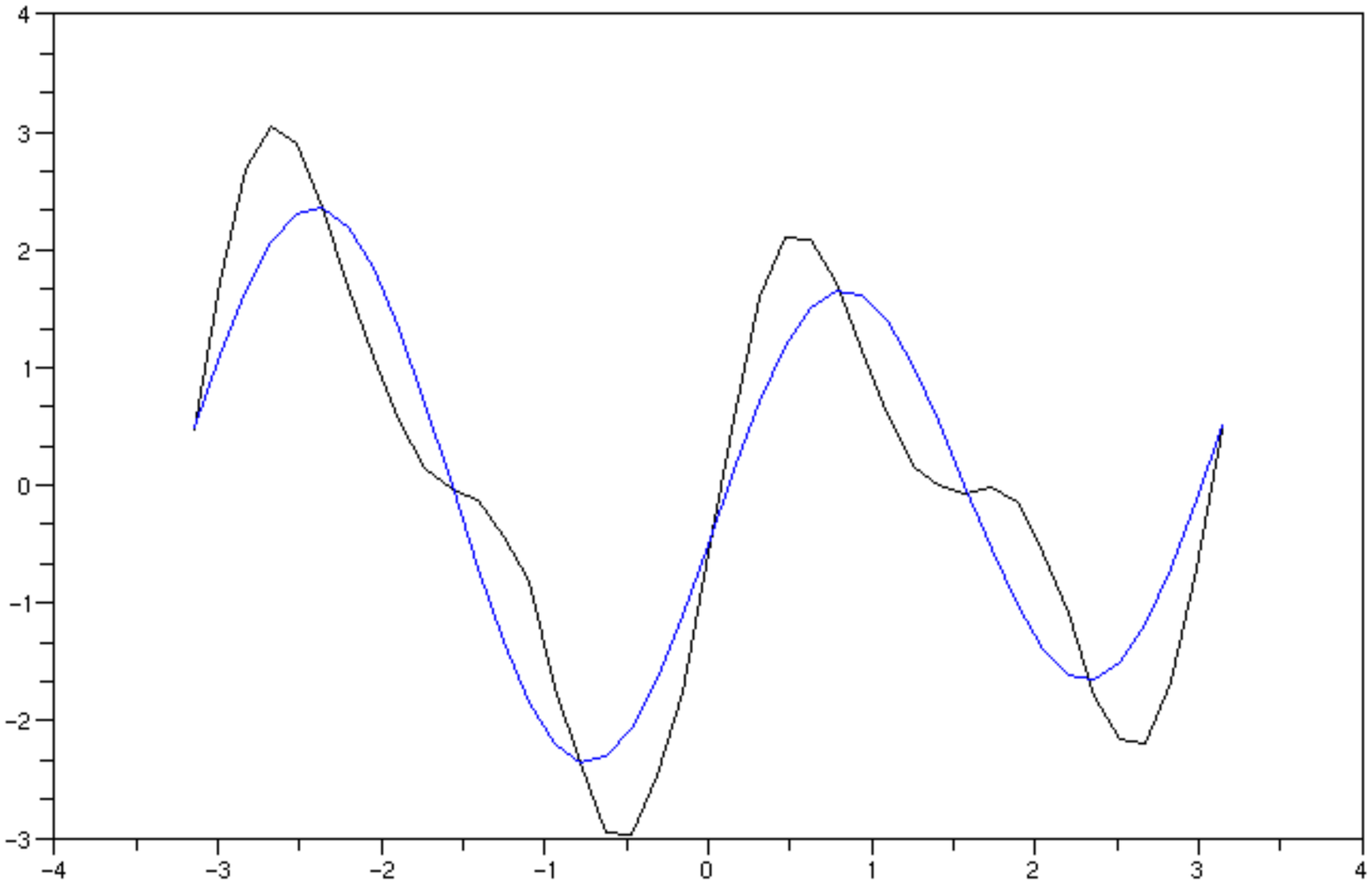}}
\caption{Comparison of the mean force profiles,  $\gamma=10$. Blue:
  $-F'(x)$. Black: $\lim_{t \to \infty}
\frac{1}{t} \int_0^t \frac{d\Psi_s}{dx}(x) \, ds$.}\label{fig:3}%
\end{figure}

\subsection{Discrete state space metadynamics: an analytical study}\label{sec:non_adiabatic_disc}

\subsubsection{The non-adiabatic metadynamics}

The non-adiabatic version of the discrete metadynamics introduced in
Section~\ref{sec:adiabatic_disc} consists in grouping states into a certain number of bins, and compute the 
bias not site by site, but bin by bin. More precisely let 
 $\xi: \{0, \ldots , K\} \to \{0, \ldots , B\}$ be a collective variable,
 with $B<K$:
  the bins are the
 levet sets $(\xi^{-1}(\{b\}))_{0\leq b\leq B}$. 
Let us  record the 
 current position of the walker $I_t \in \{0, \ldots, K\}$  and the local times spent
 in each bin $b \in \{0, \ldots, B\}$
 \[
   L_t(b) = \int_0^t \ind{\xi(I_s) = b}ds
 \]
so that $L_t=(L_t(0), \ldots, L_t(B)) \in \dR^{B+1}$. Let $V:\{0,
\ldots, K\} \to \dR$ be the potential function and let us define
$$\forall k \in \{1, \ldots, K\}, V_k'=V_k-V_{k-1}.$$
 The generator of the process $(L_t,I_t)_{t \ge 0}$ is given by: for any test
 function $f: \dR^{B+1} \times \{0, \ldots , K\} \to \dR$, for all
 $(l,k) \in \dR^{B+1} \times \{0, \ldots , K\}$,
\[
  \begin{split}
    Lf(l,k) =& \partial_{l_k} f(l,k)  \\
  & + 1_{\{k<K\}} \exp\left(-\beta ( \gamma [l(\xi(k+1)) - l(\xi(k))] +V'_{k+1})\right) \left(f(l,k+1)- f(l,k)\right) \\
  & + 1_{\{k>0\}} \exp\left(+\beta ( \gamma [l(\xi(k)) - l (\xi(k-1))]
    + V'_k) \right) \left(f(l,k-1) - f(l,k)\right).
				  \end{split}
\]
Just as in Section~\ref{sec:adiabatic_disc}, we could instead define a process acting on the differences
of local times --- since the expression of the generator is messier we omit it. 

One could hope that this process, similarly to the simpler one we described in the
Section~\ref{sec:adiabatic_disc}, "learns the free energy profile", in the sense that the difference
of local times $\gamma (L_t(b) - L_{t}(b'))$ (for $(b,b') \in \{0,\ldots,B\}^2$) fluctuates around the difference of free energy
\[ A_{b'} - A_{b} = \frac{1}{2\beta} \ln\PAR{
    \frac{\pi(\xi^{-1}(\{b\}))}{\pi(\xi^{-1}(\{b'\}))}}\]
 where for  $k \in \{0,\ldots,K\}$, $\pi(\{k\})=\frac{e^{-2\beta V_k}}{\sum_{j=0}^Ke^{-2\beta V_j}}$.
Should this hold, then the free energy profile $(A_b)_{0 \le b \le B}$ could then 
be recovered by taking time averages.

Unfortunately this is not the case. To see why, we consider a simple case 
with $4$~sites $\{0,1,2,3\}$ and two bins: $B_-=\{0,1\}$ and $B_+=\{2,3\}$. We define
the potential $V$ by $V_1 = V_2$, $V_1 - V_0 = D_- >0$ and 
$V_2 - V_3 = D_+ >0$: $V$ has two wells in $0$ and $3$, separated 
by a saddle $\{1,2\}$. We let $$\lambda_\pm = \exp(-\beta D_\pm)$$ be
the jump rates from the bottom of the wells to the saddle. 

The target measure is proportional to $\exp(-2\beta V)$; the weights 
it gives to each bin are
\begin{align*}
  \pi(B_-) =& \frac{1+\lambda_-^{-2}}{2+\lambda_-^{-2} + \lambda_+^{-2}} 
  &
  \pi(B_+) =& \frac{1+\lambda_+^{-2}}{2+\lambda_-^{-2} + \lambda_+^{-2}}
\end{align*}
so the difference of local times $X_t=\gamma ( L_t(+) - L_{t}(-))$ should fluctuate around
\begin{equation}\label{eq:FEdiff}
A_--A_+= \frac{1}{2\beta} \ln\PAR{
  \frac{1+\lambda_+^{-2}}{1+\lambda_-^{-2}}}.
\end{equation}

For this model however, the heuristic picture is very different: 
\begin{equation}\label{eq:approx}
\text{\emph{the
  trajectorial average of $(X_t)$ behaves like $\gamma\left(\frac{1}{\lambda_+} -
  \frac{1}{\lambda_-}\right)$,}}
\end{equation} at least when
$\gamma$ is sufficiently large, and this limit is indeed different
from the free energy difference~\eqref{eq:FEdiff}.
We justify this claim with an informal reasoning just below, and with 
an analytical proof in a slightly simplified setting in Section~\ref{sec:3States}.

Let us first give an idea of  the trajectories of the process in the
large $\gamma$ regime. 
  Starting from $i=0$ and $x=0$, the process:  
  \begin{enumerate}
  	\item stays in $0$ for a time $E_-$ (exponential with parameter
	  $\lambda_-$), 
	\item jumps to $1$, where it immediately jumps to $2$ (since
	  at that time, $X_t= - \gamma E_-$, with $\gamma E_-$ very large), and then quickly jumps
	  to $3$. 
  \end{enumerate}
  At this point $|X_t|$ is large. While it stays large, even if $I_t$ jumps
  to $2$ it will quickly jump back to $3$, since the jump to $1$ is 
 essentially blocked.  So the process must
  wait in  $3$ for a  time at least $E_-$, so that $X_t$ has time to decrease
  back to $0$. 
  After this, the process waits for an exponential time $E_+$ (with parameter~$\lambda_+$) until it jumps to $2$. Once more it jumps immediately
  to $1$ and then to $0$, and must stay there (with brief visits to $1$) until 
  $X_t$ comes back to $0$. 
  During such a cycle, of duration $T = 2(E_-+E_+)$,  $\int_0^T  X_s ds$ is
  approximately $- \gamma E_-^2 + \gamma E_+^2$. In expectation this yields
  $(2\gamma)(\lambda_+^{-2} -\lambda_-^{-2})$. By the ergodic theorem, since
  $\esp{T} = 2(\lambda_+^{-1} + \lambda_-^{-1})$, this gives 
  the claimed approximation~\eqref{eq:approx}: 
  $\frac{1}{t} \int_0^t X_s \, ds \approx_{t \to \infty} \gamma \left(\frac{1}{\lambda_+} -
    \frac{1}{\lambda_-} \right)$. 

  We tested this claim numerically for various values of $\gamma$, $\lambda_\pm$, 
  see Figure~\ref{fig:fourStates} for an illustration. 

  \begin{figure}
    {\centering
    \includegraphics{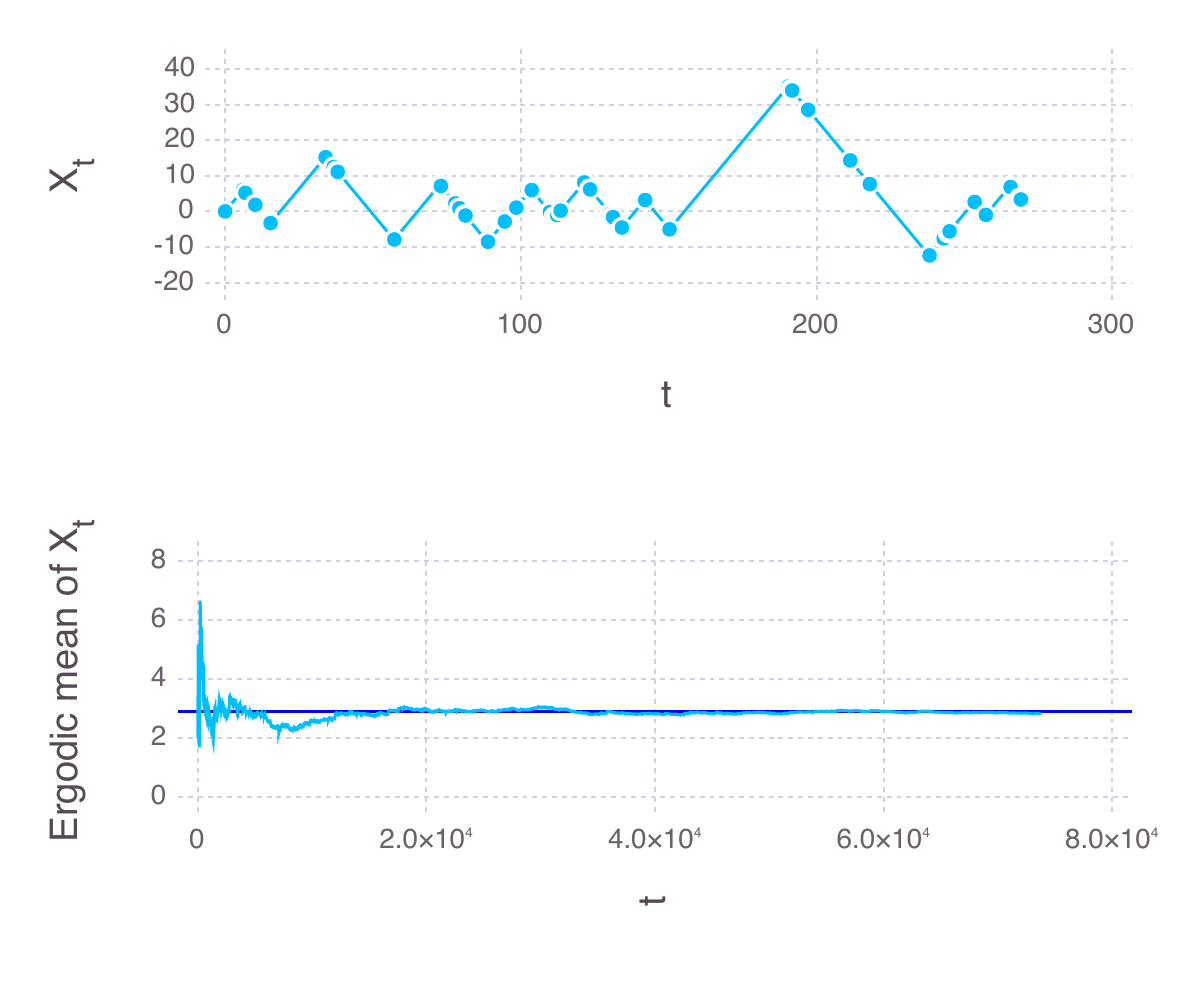}
    \par}
    An example trajectory for the four-states, two bins model, in the 
    uneven double well potential $V$ given by   $V_0=0$, 
    $V_1=V_2=2$ and $V_3=0.5$, with $\gamma = 1$ and $\beta=1$. Upper plot: the trajectory of $X_t$. The dots represent times when $I_t$ jumps. Lower plot: the ergodic mean $\frac{1}{t} \int_0^t X_s ds$. The horizontal 
    line is $\gamma\left(\frac{1}{\lambda_+} - \frac{1}{\lambda_-}\right)$. 
    \caption{The four-states, two-bins model.}
    \label{fig:fourStates}
  \end{figure}

Rather than trying to make this informal picture more precise, we simplify 
the model a bit more in the next section so that it becomes analytically tractable. 

\subsubsection{Analytical results on a simplified non-adiabatic metadynamics}
\label{sec:3States}
In order to be able to write an explicit expression for the invariant measure, we modify
the model with 4 sites and 2 bins considered in the previous section. We idealize the saddle $\{1,2\}$ and replace it by an abstract state; 
we relabel the states to $\{-,0,+\}$ where $-$ and $+$ are the two wells and $0$ is the saddle
between them. 
Let us then define the following generator: 
\begin{equation}\label{eq:Lsimp}
\begin{aligned}
  L^{simp}f(x,i) &= \gamma i\partial_x f(x,i) 
    + 1_{\{i=\pm\}}\lambda_\pm(f(x,0)-f(x,\pm))\\
&  \quad  + 1_{\{i=0\}}\left(\frac{e^{\beta x}}{\lambda_-}(f(x,-)-f(x,0))
    + \frac{e^{-\beta x}}{\lambda_+}(f(x,+)-f(x,0))\right).
\end{aligned}
\end{equation}
The measure 
\begin{equation}\label{eq:musimp1}
d\mu^{simp}(x,i)  = \delta_{-}(di)\mu_-(x)dx 
+ \delta_{0}(di) \mu_0(x)dx
+ \delta_{+}(di) \mu_+(x)dx
\end{equation}
is invariant for the dynamics with generator $L^{simp}$ if and only if
\[
  \left\{\begin{aligned}
      \gamma\mu_-'(x)-\lambda_-\mu_-(x)+\frac{e^{\beta x}}{\lambda_-}\mu_0(x) =& 0 \\
      \lambda_-\mu_-(x)+\lambda_+\mu_+(x) 
        - \left(\frac{e^{\beta x}}{\lambda_-}+\frac{e^{-\beta x}}{\lambda_+}\right)
	\mu_0(x) =& 0 \\
      \gamma\mu_+'(x)+\lambda_+\mu_+(x)-\frac{e^{-\beta x}}{\lambda_+}\mu_0(x) =& 0.
   \end{aligned}\right.
 \]
This system is equivalent to
\[
  \left\{\begin{aligned}
   \mu_0(x) =& \frac{\lambda_+\lambda_-^2\mu_-(x)+\lambda^2_ 
           + \lambda_-\mu_+(x)}{\lambda_+e^{\beta x}+\lambda_-e^{-\beta x}}\\
   \gamma \mu_-'(x) = \gamma \mu_+'(x) 
         =& \frac{\lambda_-^2e^{-\beta x}\mu_-(x)-\lambda_+^2e^{\beta x}\mu_+(x)}{\lambda_+e^{\beta x}+\lambda_-e^{-\beta x}}.
\end{aligned}\right.
\]
The last equation implies that $\mu_-(x)-\mu_+(x)$ is constant and for the invariant measure to be finite that $\mu_-(x)=\mu_+(x)$. Hence 
\[
  \left\{\begin{aligned}
  \mu_-(x) =& \mu_+(x)\\
  \mu_0(x) =& \frac{\lambda_+\lambda_-^2+\lambda^2_+\lambda_-}{\lambda_+e^{\beta x}+\lambda_-e^{-\beta x}}\mu_+(x)\\
  \gamma \mu_+'(x) =& \frac{\lambda_-^2e^{-\beta x}-\lambda_+^2e^{\beta x}}{\lambda_+e^{\beta x}+\lambda_-e^{-\beta x}}\mu_+(x).
\end{aligned}\right.
\]
Setting $\tilde{\mu}(y)=\mu_+\left(\frac{\ln(y)}{\beta}\right)$, one has
$$(\ln \tilde{\mu})'(y)=\frac{1}{\gamma\beta}\frac{\lambda_-^2-\lambda_+^2y^2}{\lambda_+y^3+\lambda_-y}=\frac{1}{\gamma\beta}\left(\frac{\lambda_-}{y}-\frac{\lambda_+(\lambda_++\lambda_-)y}{\lambda_+y^2+\lambda_-}\right).$$
One concludes that 
\begin{equation}\label{eq:musimp2}
  \mu_+(x) 
  = \frac{1}{C}e^{\frac{\lambda_-}{\gamma} x}
    (\lambda_+ e^{2\beta x}+\lambda_-)^{-\frac{\lambda_++\lambda_-}{2\beta \gamma}}
  = \frac{1}{C}\left(
    \lambda_+ e^{\frac{2\beta \lambda_+}{\lambda_++\lambda_-}x}
    + \lambda_- e^{-\frac{2\beta \lambda_-}{\lambda_++\lambda_-}x}
  \right)^{-\frac{\lambda_++\lambda_-}{2\beta \gamma}}
\end{equation}
where the constant $C$ is chosen so that $\mu$ is a probability measure.  
To simplify notation define $x_0 = \frac{1}{2\beta}
\ln(\lambda_-/\lambda_+)=\frac1 2 \left( D_+ - D_-\right)$, $d  = \lambda_+ - \lambda_-$ and $s = \lambda_++\lambda_-$. 
Then 
\[
  \lambda_+ \exp\PAR{\frac{2\beta\lambda_+ x}{\lambda_+ + \lambda_-}}
  +  \lambda_- \exp\PAR{\frac{- 2\beta\lambda_- x}{\lambda_+ + \lambda_-}}
  = \exp\PAR{\frac{\beta dx}{s}} 2\cosh(\beta(x-x_0)) \sqrt{\lambda_-\lambda_+}
\]
The density of the marginal of the invariant measure $\mu^{simp}$
on the $x$ variable is proportional to: 
\[
  f(x) = \exp\PAR{ -\frac{d(x-x_0)}{2\gamma} }
  (2\cosh(\beta(x-x_0)))^{-s/(2\beta\gamma)}
  \PAR{ 2 + \frac{s \sqrt{\lambda_-\lambda_+}}{2\cosh(\beta(x-x_0))}}. 
\]
\begin{lem}[Asymptotics for $\int x d\mu^{simp}(x,i)$]
 Assume that $D_+\neq D_-$.  Then, the average of $x$
under the
 invariant measure $d\mu^{simp}(x,i)$ defined by~\eqref{eq:musimp1}--\eqref{eq:musimp2} behaves like $\gamma\PAR{e^{\beta D_+} - e^{\beta D_-}}$ when 
  this quantity is large in absolute value, in the sense that
  \[
    \frac{1}{\gamma\PAR{e^{\beta D_+} - e^{\beta D_-}}} \int x d\mu^{simp}(x,i) 
    \to 1,
  \]
  when any one of the parameters $\beta$, $\gamma$, $D_+$ or $D_-$ goes to 
  infinity while the others stay fixed. 
\end{lem}
\begin{rem} In the 
  symmetric case $D_- = D_+$, the mean value of $x$ is zero. 
\end{rem}
\begin{proof}
 We wish to compute the limit  of \( \displaystyle m(\beta,\gamma,D_+,D_-) =  \frac{\int xf(x) dx}{ \int f(x) dx}\) 
 when one of the parameters $\beta$, $\gamma$, $D_+$ or $D_-$  goes to infinity. 
We 
change variables by 
setting $x = x_0 + \gamma\PAR{\frac{1}{\lambda_+} - \frac{1}{\lambda_-}} y
 = x_0 - \frac{\gamma d}{\lambda_-\lambda_+} y$. 
Let $g(y) = f\PAR{x_0 - \frac{\gamma d}{\lambda_-\lambda_+} y}$. 
Then 
\[
  m(\beta,\gamma,D_-,D_+) = x_0 - \frac{\gamma d}{\lambda_-\lambda_+} \frac{\int yg(y) dy}{\int g(y)dy},
\]
where 
\[ 
  g(y) = 
  \exp\PAR{ \frac{d^2}{2\lambda_-\lambda_+} y}
  \PAR{2\cosh \PAR{ \frac{\beta\gamma d}{\lambda_-\lambda_+} y }}^{-s/(2\beta\gamma)}
  \PAR{2 + \frac{ s \sqrt{\lambda_-\lambda_+}}
                {2\cosh\PAR{\frac{\beta \gamma d}{\lambda_-\lambda_+} y}}%
      }. 
\]
We thus have, since $\gamma(e^{\beta D_-}-e^{\beta D_+})=\frac{\gamma d}{\lambda_-\lambda_+}$:
\[ 
  \frac{1}{\gamma\PAR{e^{\beta D_+} - e^{\beta D_-}}} m(\beta,\gamma,D_-,D_+)
  = 
  \frac{D_+-D_-}{2 \gamma\PAR{e^{\beta D_+} - e^{\beta D_-}}} + \frac{\int yg(y) dy}{\int g(y)dy}.\]

For any of the limits we consider, the first term in the right-hand
side goes to zero, and for any fixed $y$, $\frac{\gamma d}{\lambda_-\lambda_+}$
goes to infinity so that
\begin{align*}
  g(y) &\sim 2 
  \exp\PAR{ \frac{d^2}{2\lambda_-\lambda_+} y}
  \PAR{\exp\PAR{ - \frac{ s\abs{d}}{2\lambda_-\lambda_+} \abs{y} }} \\
  &\sim \begin{cases} 
    2\exp\PAR{ - \frac{\abs{d}}{\max(\lambda_-,\lambda_+)}  \abs{y}}, & \text{ if } y>0 \\
    2\exp\PAR{ - \frac{\abs{d}}{\min(\lambda_-,\lambda_+)}\abs{y}}, & \text{ if } y<0.
  \end{cases}
\end{align*}
Let us write $h(y)$ for the last function appearing in these equivalences. 

Consider first the $\gamma\to\infty$ limit. The function $h(y)$ 
 does not depend on $\gamma$, and a 
direct computation shows that:
\begin{align*}
  \int y h(y) dy 
  =& \frac{2}{\abs{d}^2}\PAR{\max(\lambda_-,\lambda_+)^2 - \min(\lambda_-,\lambda_+)^2} \\
  =& \frac{2}{\abs{d}}\PAR{\max(\lambda_-,\lambda_+) + \min(\lambda_-,\lambda_+)} 
  =\int h(y) dy. 
\end{align*}
Since $yg(y)$ is easy to dominate, we get  by Lebesgue's theorem:
\[ 
  \frac{1}{\gamma\PAR{e^{\beta D_+} - e^{\beta D_-}}} m(\beta,\gamma,D_-,D_+)
\to \frac{\int yh(y)dy}{\int h(y)dy} = 1.
\]

For the other cases, where one of  $\beta$, $D_-$ or $D_+$ converges to $\infty$ 
with all other parameters fixed, we compute one step further to get
\begin{align*}
  g(y) \sim h(y) 
  \to 2\ind{y>0} \exp\PAR{ - y}. 
\end{align*}
We apply Lebesgue's theorem once more:  $\int yg(y)dy/\int g(y)dy$ converges to
$\int_0^\infty y e^{-y} dy/\int_0^\infty e^{-y}dy = 1$, 
and the result follows. 
\end{proof}

The fact that the trajectorial average of $X_t$ converges to $\int xd\mu^{simp}(x,i)$ 
is guaranteed by the following result. 
\begin{lem}
The process $(X_t,I_t)$ with generator $L^{simp}$ defined
by~\eqref{eq:Lsimp} is  exponentially ergodic (in the sense of
Equation~\eqref{eq:experg}) with respect to the
measure $\mu^{simp}$ defined
by~\eqref{eq:musimp1}--\eqref{eq:musimp2}. Furthermore we have the convergence of
the ergodic mean: 
\[ \frac{1}{t} \int_0^t X_s ds \xrightarrow{\text{a.s.}} \int x d\mu^{simp}(x,i).\]  
\end{lem}
\begin{proof}
We build a function $W$ that satisfies
\[ L^{simp} W(x,i) \leq - \delta W(x,i) + M \ind{\abs{x} \leq C}.\]
If the process starts in $\{-,0,+\}\times [C,\infty)$, it cannot
reach $\{-,0,+\}\times (-\infty,-C]$ without coming back in the compact $\{-,0,+\} \times [-C,C]$. 
Therefore we may define $W$ separately on $\{-,0,+\}\times \dR_+$ and $\{-,0,+\}\times \dR_-$. 
We only do the first case, the second being similar. 
On $\dR_+\times \{-,0,+\}$, let $W(x,i) = C_ie^{\alpha x}$ where the constants
satisfy $0<C_-<C_0<C_+$ and
\[ \lambda_-\left( \frac{C_0}{C_-} - 1\right) < \gamma\alpha < \lambda_+ \PAR{1 - \frac{C_0}{C_+}}.\]
Such a choice is always possible by choosing $C_0=1$, $C_+=2$ and $C_-$ close enough to $1$. 
A simple computation yields
\begin{align*}
  L^{simp}W(x,+) =& \gamma\alpha C_+ e^{\alpha x} + \lambda_+(C_0e^{\alpha x} - C_+e^{\alpha x}) \\
  =& - \PAR{ \lambda_+\PAR{1- \frac{C_0}{C_+}} - \gamma\alpha} W(x,+), \\
  L^{simp}W(x,0) =& \frac{e^{\beta x}}{\lambda_-}(C_- - C_0) e^{\alpha x} + \frac{e^{-\beta x}}{\lambda_+} (C_+ - C_0) e^{\alpha x} \\
  =& - \PAR{ \frac{e^{\beta x}}{\lambda_-} \PAR{1 - \frac{C_-}{C_0}} 
            - \frac{e^{-\beta x}}{\lambda_+}\PAR{\frac{C_+}{C_0} -1}} W(x,0) \\
  L^{simp}W(x,-) =& - \gamma\alpha C_- e^{\alpha x} + \lambda_-\PAR{C_0 e^{\alpha x} - C_- e^{\alpha x}} \\
          =& - \PAR{\gamma\alpha - \lambda_-\PAR{\frac{C_0}{C_-} - 1}} W(x,-)
\end{align*}
so that $L^{simp}W\leq - \delta W + C \ind{\abs{x} \leq C}$ for appropriate choices of $C<\infty$ and $\delta>0$. 

The exponential ergodicity follows by classical arguments that will be
developed below  in the 
proof of Theorem~\ref{thm:cvg_vt}: the process is irreducible, aperiodic
and the compact set $\{-,0,+\}\times [-C,C]$ is petite, so 
the ergodicity follow from~\cite[Theorem 5.2 (c)]{DMT95}.

The fact that the ergodic mean converges then follows from the fact that $(x,i)\mapsto x$
is locally bounded, so that Equation~\eqref{eq:finite_along_trajectories} holds, and
we can follow the same proof as for Theorem~\ref{th:ergo}. 
\end{proof}

\section{Mathematical analysis of the adiabatic discrete case}\label{sec:proof_discrete}
We prove the main mathematical results, Theorems~\ref{thmexperg}, \ref{th:ergo}  and
\ref{thm:clt}, as follows. 
We start in Section~\ref{sec:invariant} by checking the explicit expression of the invariant
measure.
In Section~\ref{sec:separated},  we state that there exists an
explicit Lyapunov
function~$W$ such that, for ``nice'' starting points $(x,i)$, $P_tW(x,i) \leq
\kappa W(x,i)$ (where $P_t$ is the transition operator associated with
the infinitesimal generator~\eqref{eq:generator_metadynamics_flat}). This is the content of Theorem~\ref{thm:Vdecreases}, the
proof of which is postponed to Section~\ref{sec:Vdecreases} for ease of reading. 
Using this, we prove in Section~\ref{sec:expMoments} that $W$ is a Lyapunov
function for an auxiliary, discrete time Markov chain, from which we
deduce the control on return times stated in Theorem~\ref{thm:expMoments_continuous}. 
We proceed to the proof of the main results,  Theorems~\ref{thmexperg}, \ref{th:ergo} and \ref{thm:clt}, 
in Section~\ref{sec:prf_clt}. 
Finally, we establish in Section~\ref{sec:auxiliary_process} precise bounds
for  the process appearing in the Ray-Knight representation already stated in Section~\ref{sec:RayKnight}; in turn, these
bounds enable us to prove the  key intermediate result
Theorem~\ref{thm:Vdecreases} in Section~\ref{sec:Vdecreases}. 

\subsection{The invariant measure}
\label{sec:invariant}
We prove here the explicit expression of the invariant measure
announced in Proposition~\ref{prop:mu}.
  Let $f:\dR^K\times\{0,\hdots,K\}\to\dR$ be a nice function (say, for each $k\in\{0,\hdots,K\}$, $x\mapsto f(x,k)$ is smooth with compact support)
  and let us compute $\int Lf d\mu$. 
  \begin{align*}
    Z\times \int Lf(x,k) d\mu(x,k)
    =& \gamma\int\sum_{0<k\leq K} \partial_kf (x,k) \exp( - \sum_j g_j(x_j)) dx_1\cdots dx_K
    \\
    &\quad 
      -  \gamma \int \sum_{0\leq k< K}\partial_{k+1}f (x,k)  \exp( - \sum_j g_j(x_j)) dx_1\cdots dx_K \\
    &\quad 
    + \int\sum_{0\leq k < K}  e^{-\beta (x_{k+1}+A'_{k+1})} f(x,k+1) dm(x)
      - \int\sum_{0\leq k< K}  e^{-\beta( x_{k+1}+A'_{k+1})} f(x,k) dm(x) \\
    &\quad +
    \int\sum_{0<k\leq K} e^{\beta(x_{k}+A'_k)} f(x,k-1) dm(x)
      - \int\sum_{0<k\leq K} e^{\beta (x_{k}+A'_k)} f(x,k) dm(x).
    \end{align*}
    In the first two terms, we swap the sum and integral signs and integrate by parts,
    respectively on the $x_k$ and  $x_{k+1}$  variables. In the third and fifth term
    we shift indices to get:
  \begin{align*}
    Z\times \int Lf(x,k) d\mu (x,k)
    &= \gamma\int\sum_{0<k\leq K} f (x,k)g_k'(x_k)  dm(x)
      -  \gamma\int \sum_{0\leq k< K} f(x,k) g_{k+1}'(x_{k+1}) dm(x) \\
    &\quad 
    + \int\sum_{0< k \leq K}  e^{-\beta (x_{k}+A'_{k})} f(x,k) dm(x)
      - \int\sum_{0\leq k< K}  e^{-\beta( x_{k+1}+A'_{k+1})} f(x,k) dm(x) \\
    &\quad +
    \int\sum_{0\leq k< K} e^{\beta(x_{k+1}+A'_{k+1})} f(x,k) dm(x)
      - \int\sum_{0<k\leq K} e^{\beta (x_{k}+A'_k)} f(x,k) dm(x) \\
    &=\int\sum_{0<k\leq K} f (x,k)\PAR{
      \gamma g_k'(x_k) +  e^{-\beta (x_{k}+A'_{k})} - e^{\beta (x_{k}+A'_k)}
      }	dm(x)  \\
    &\quad   +  \int \sum_{0\leq k< K} f(x,k) \PAR{ 
      -\gamma  g_{k+1}'(x_{k+1})  - e^{-\beta( x_{k+1}+A'_{k+1})} 
      + e^{\beta(x_{k+1}+A'_{k+1})} }
    dm(x)
    \end{align*}
    by regrouping terms in the sums. 
    Since $\gamma g'_k(y) = e^{\beta (y+A'_k)} - e^{-\beta ( y+A'_k)}$,
    the terms between brackets vanish, so $\int Lf(x,k) d\mu(x,k) = 0$ and
    $\mu$ is invariant.

\subsection{Bounds for \texorpdfstring{$t$}{t}-separated starting points}
\label{sec:separated}

\begin{defi}[Separated configuration and plateau]
  \label{def:separated}
 Let $A>a>0$ and $t >0$.
\begin{itemize}
   \item A configuration $(x,i)\in \dR^K\times\{0,\ldots,K\}$ (or a vector $x\in\dR^K$) is called \emph{$(t,a,A)$-separated}
  if the components of $x$ are either very large or very small, in the sense that
  \[
    \forall j\in\{1,\ldots,K\}, \quad
    \abs{x_j} \leq a t 
    \text{ or }
    \abs{x_j} >At,
  \]
  and at least one of the $|x_j|$ is larger than $At$. 
\item A connected set of sites 
  $I=\{l,l+1, \ldots,r-1\}\subset\{0,\ldots,K\}$ (for $l \le r-1$) is called a
  \emph{plateau} of a $(t,a,A)$-separated vector $x\in\dR^K$ if \begin{itemize}
\item $\abs{x_j}<at$ for all $l<j<r$,
\item when $l\ge 1$, $x_l>At$,
\item when $r\le K$, $x_r<-At$. \end{itemize}
\end{itemize}
 \end{defi}
For a vecteur $x \in \dR^K$,
let us consider, for $k \in \{0, \ldots, K\}$, $l_k=\sum_{m=1}^k x_m$: a plateau is a connected set
of sites where $l$ does not vary too much, with on the boundaries of the
plateau a
drop to lower values. In the following, we will refer to sites
which are not on a plateau as wells. We refer to Figure~\ref{fig:sand} for a schematic representation of
a separated vector, with its plateaux. 
\begin{lem}
    Any $(t,a,A)$-separated vector $x\in\dR^K$ admits a plateau.
\end{lem}
\begin{proof}We proceed by induction on $K$.
If $K=1$, then either $x_1<-At$ and $I=\{0\}$ is a plateau or $x_1>At$ and $I=\{1\}$ is a plateau. 
Let us assume that the result holds when $1\le K<\bar{K}$ with $\bar{K}\ge 2$ and let $(x,i)\in \dR^{\bar{K}}\times\{0,\ldots,\bar{K}\}$ be $(t,a,A)$-separated.
Let $\underline{j}=\min\{j\ge 1:|x_j|>A t\}$. If $x_{\underline{j}}<-At$, then $I=\{0,\ldots,\underline{j}-1\}$ is a plateau. If $x_{\underline{j}}>At$, then either $\max_{\underline{j}+1\le j\le \bar{K}}|x_j|\le a t$ and $I=\{\underline{j},\ldots,\bar{K}\}$ is a plateau or $\max_{\underline{j}+1\le j\le \bar{K}}|x_j|> At$. In the latter case, $(x_{\underline{j}+1},\ldots,x_{\bar{K}})$ is $(t,A,a)$-separated and admits a plateau by the induction hypothesis at rank $\bar{K}-\underline{j}<\bar{K}$. The translation by $\underline{j}$ of the indices of this plateau form a plateau of $(x_0,\ldots,x_K)$ since $x_{\underline{j}}>At$.
\end{proof}

Starting from a $(t,a,A)$-separated configuration $(x,i)$, the process with generator~\eqref{eq:generator_metadynamics_flat} is highly 
predictable on the $t$ time scale:
if $a$ is small enough and $A$ large enough, 
then during a time $t$, the process will (almost always) either jump
to to lower energy sites or jump over small barriers of size $at$
but (almost never) over large barriers of size $At$. This is because
over the time $t$, the process cannot fill the wells of the initial configuration.
Let us state this fact more precisely. For a configuration
$x$, let 
\begin{equation}\label{eq:antiderivative}
(l_0,\ldots,l_K)=(l_0,l_0+x_1,x_1+x_2,\ldots,l_0+x_1+\ldots+x_K)
\end{equation}
be a discrete antiderivative of $x$ (where $l_0$ is completely
arbitrary and plays no role in the dynamics and in the value of $S(x)$) and
let us recall the definition of the function $S$ introduced
in~\eqref{eq:S}, written here in terms of the antiderivative $l$:
\begin{equation}\label{eq:Sl}
  S(x) = \sum_{i=0}^K \left(\left(\max_{0\le j\le K} l_j\right) - l_i\right).
\end{equation}
This represents the amount of "computational sand" needed to fill all the gaps in $x$
(see Figure~\ref{fig:sand}). Notice that 
\begin{equation}
   \label{eqborns} S(x)\le K \sum_{k=1}^K|x_k|.
\end{equation}

\begin{figure}
  \centering
\begin{tikzpicture}
    \begin{axis}[ybar stacked,
        bar width=0.8]
    \addplot coordinates
        {(0,3.1) (1,3) (2,3.2) (3,1) (4,4) (5,3.8) (6,2) (7,2.2) (8,2.1) (9,2)};
    \end{axis}
\end{tikzpicture}
\hspace{2em}
\begin{tikzpicture}
    \begin{axis}[ybar stacked,
        bar width=0.8]
    \addplot coordinates
        {(0,3.1) (1,3) (2,3.2) (3,1) (4,4) (5,3.8) (6,2) (7,2.2) (8,2.1) (9,2)};
    \addplot coordinates
        {(0,0.9) (1,1) (2,0.8) (3,3) (4,0) (5,0.2) (6,2) (7,1.8) (8,1.9) (9,2)};
    \end{axis}
\end{tikzpicture}

{\small
On both pictures the blue bars represent the profile
$l_0,l_1,...,l_K$, that satisfies $x_k=l_k - l_{k-1}$. The value of
$l_0$ can be chosen arbitrarily (energy landscapes are defined up to
an additive constant). On the left: a separated configuration with two plateaus: $\{0,1,2\}$
and $\{4,5\}$ (for $(t,a,A)=(1,1/2,3/2)$, see
Defintion~\ref{def:separated}).  On the right:  The function $S(x)$ is
the sum of the lengths in red (see~\eqref{eq:S}). 
}
\caption{Schematic representation of the plateaus of a separated
  vector, and of the function $S$.}\label{fig:sand}
\end{figure}

Starting from  a $(t,a,A)$-separated configuration  $(x,i)$, then
between time $0$ and $t$, the process $(I_t)_{t \ge 0}$
will behave essentially essentially as follows: if $i$ is in a well,
$I_t$ will stay in this well, if $i$ is on a plateau, $I_t$ will go
very quickly in one of the wells on the boundaries of the plateau, and stay there. In both cases, the
process $I_t$ will stay for a long time in a well and ``fill it up'',
until time $t$: this will thus decrease the values of $S(X_t)$. This
is quantified in the following theorem, and it yields a natural Lyapunov
function for the process $(X_t,I_t)$.
\begin{thm}
  \label{thm:Vdecreases}
Let us consider the process $(X_t,I_t)$ with
generator~\eqref{eq:generator_metadynamics_flat}. Let us denote by
$P_t$ the associated transition operator, and $(x,i)$ the initial
condition of the process.
  There exist positive numbers $a<A$, $\chi$, $t_0$ and $\kappa\in(0,1)$, 
  such that  for all $t\geq t_0$, for all starting points $(x,i)\in\dR^K\times\{0,\hdots,K\}$, 
  \[
  (x,i) \text{ is }(t,a,A)\text{-separated}
 \implies 
    P_t W(x,i) \leq \kappa W(x,i)
  \]
where $W(x,i) = \exp(\chi S(x))$.
\end{thm}
The proof of this theorem is  given below in
Section~\ref{sec:Vdecreases}. It draws on the Ray-Knight description of the process found in
\cite{TV11} and already introduced in Section~\ref{sec:RayKnight}.

\subsection{Exponential moments for return times --- Theorem~\ref{thm:expMoments_continuous}}
\label{sec:expMoments}
Let $t_0$, $a$, $A$ be given by Theorem~\ref{thm:Vdecreases}. For $k
\in \{0, \ldots, K\}$,
let $t_k = t_0(A/a)^k$, and let $$\forall k
\in \{0, \ldots, K-1\}, \, \cS_k = \left\{ x \in \dR^K, x \text{ is $(t_k,a,A)$-separated}\right\}.$$ 
Finally let us introduce the compact set $\cK \subset \dR^K$: 
$$\cK=\left\{x \in \dR^K, \, \max \abs{x_j} \leq At_{K-1}\right\}.$$
One has $\dR^K = \PAR{\bigcup_{k=0}^{K-1} \cS_k} \cup \cK$. Indeed, if
$x$ is not in $\cK$, then one of its $K$ coordinates is larger in
absolute value than $At_{K-1}$ and since the $K$ intervals $]at_0,At_0]$, $]at_1,At_1]$,\ldots, $]at_{K-1},At_{K-1}]$ are disjoint, there is an index $k\in\{0,\ldots,K-1\}$ such that $]at_k,At_k]$ contains no coordinate of $x$ so that $x\in\cS_k$.

Define the discrete time Markov chain $(Y_n,J_n)_{n \in \dN}$
with values in $\dR^K \times \{0, \ldots, K\}$ as follows: at step~$n$, 
from $(Y_n,J_n) = (y,j)$, compute
$\underline{k}(y)=\min\{k\in\{0,\hdots,K-1\}:y\in \cS_k\}$ with
convention $\min\emptyset=K$ (applied when
$y\in\cK\setminus\bigcup_{k=0}^{K-1} \cS_k$), run the continuous
time process $(X_t,I_t)_{t \ge 0}$ with
generator~\eqref{eq:generator_metadynamics_flat} starting from
$(X_0,I_0)=(y,j)$ for a time $t_{\underline{k}(y)}$ and then set
$(Y_{n+1}, J_{n+1}) = (X_{t_{\underline{k}(y)}},
I_{t_{\underline{k}(y)}})$. Notice that for any $y \in \dR^K$,
$t_{\underline{k}(y)} \in \{t_0, t_1, \ldots, t_K\}$. Let us denote by $P$ the transition operator for 
$(Y_n,J_n)_{n \in \dN}$. 

\begin{lem}\label{lem:YnJn}
The function $W$ defined in Theorem~\ref{thm:Vdecreases} is a Lyapunov function for
the chain $(Y_n,J_n)$: there exist two constants $C<\infty$ and
$\kappa<1$ such that, for all $ (y,j) \in \dR^K \times \{0, \ldots K\}$
\begin{equation}
  \label{eq:lyapY}
  PW(y,j) \leq \kappa W(y,j) + C\ind{\cK}(y,j).
\end{equation}
Moreover, the  return time to $\cK$ has exponential moments: 
there exists $\tilde{\delta}>0$ and $\tilde{C}<\infty$ such that, for
all  $ (y,j) \in \dR^K \times \{0, \ldots K\}$
  \begin{equation}\label{eq:returntime} \esp[(y,j)]{e^{\tilde{\delta}
        \tilde{\tau}_\cK}} \leq \tilde{C} W(y,j)
\end{equation}  
where $\tilde{\tau}_\cK=\inf \{n \ge 1: \, Y_n \in \cK \}$ is the
first entry time in $\cK \times \{0, \ldots, K\}$ for the chain $(Y_n,J_n)$. 
\end{lem}

\begin{proof}
  The fact that \eqref{eq:lyapY} holds when $y\in\bigcup_{k=0}^{K-1} \cS_k$ is a straightforward consequence of
  Theorem~\ref{thm:Vdecreases} and the definition of the chain
  $(Y_n,J_n)$. Otherwise, when $y\in\cK\setminus\bigcup_{k=0}^{K-1} \cS_k$, for
  all $k\in\{1,\hdots,K\}$, $|Y_1(k)|\le |y_k|+ t_K\le At_{K-1}+ t_K$,
  so that, by \eqref{eqborns}, $W(Y_1,J_1)\le e^{K^2(At_{K-1}+ t_K)}$
  and \eqref{eq:lyapY} still holds, by choosing appropriately $C$.

  To bound the return times,  we apply \cite[Theorem 15.3.3]{MT09}
  which states that  the drift condition~\eqref{eq:lyapY}
 and the fact that $\cK$ is petite entails the desired bound~\eqref{eq:returntime} on exponential moments
 of $\tilde{\tau}_{\cK}$. Let us recall that the compact set $\cK$ is "petite"
  in the terminology of \cite{MT09} if there exists a distribution
  $a$ on $\dN$ and a non trivial measure $\nu_a$ such that, 
  if $N$ is a random variable with distribution $a$ independent of the process,
  \[
    \forall (y,j)\in\cK, \forall B \subset \dR^K \times \{0, \ldots, K\},
    \quad
    \prb[(y,j)]{(Y_N,J_N) \in B} \geq \nu_a(B).
  \]
In order to prove that $\cK$ is "petite", we will use \cite[Theorem
6.2.5]{MT09} which states that every 
  compact set is petite for a Markov chain which is open set
irreducible and which is a so called $T$-chain. Let us now prove these
two properties for $(Y_n,J_n)$.

  For any starting point $(y,j)$ and any open set $\cO\subset\dR^K\times \{0,\ldots,K\}$, 
  $\prb{\exists n, (Y_n,J_n)\in \cO}>0$, so $(Y_n,J_n)$ is open set irreducible
  (\cite[Section 6.1.2]{MT09}).
  
  The family of (constant) vector fields $(F_k)_{0\leq k \leq K}$ defined
  in Remark \ref{rem:pdmp} in Section~\ref{sec:the_model} is such that the $(F_k - F_0)_{k>0}$ are 
  linearly independent. \emph{A fortiori} the family $(F_k)_{0\leq k \leq K}$ 
 is "strong bracket generating" in the terminology of \cite{BLMZ15}, 
   so for any positive $t$ and
  starting from any point, the distribution of $(X_t,I_t)$ has a continuous component
  (that is, a component absolutely continuous with respect to the
  Lebesgue measure times the uniform measure over $\{0, \ldots, K\}$)
  by \cite[Theorem 4.2]{BLMZ15} (see also \cite[Theorem 2]{BH12}). This
  implies that $(Y_n,J_n)$ is a so-called $T$-chain (\cite[Section 6]{MT09}). 

This concludes the proof of Lemma~\ref{lem:YnJn}.
\end{proof}

To finish the proof of Theorem~\ref{thm:expMoments_continuous}, we use
a simple comparison argument. Choose an $\eta \in (0,t_0)$. 
Since $Y_n=X_{t_{\underline{k}(Y_0)} + t_{\underline{k}(Y_1)} +
  \ldots+ t_{\underline{k}(Y_{n-1})}}$ with
$t_{\underline{k}(Y_0)}\ge t_0> \eta$, one has $\tau_\cK(\eta) \le t_{\underline{k}(Y_0)} + t_{\underline{k}(Y_1)} +
  \ldots+ t_{\underline{k}(Y_{\tilde{\tau}_\cK-1})}$. Moreover, since $t_{\underline{k}(Y_0)} + t_{\underline{k}(Y_1)} +
  \ldots+ t_{\underline{k}(Y_{n-1})}\le n t_K$, one obtains $\tilde{\tau}_\cK\ge \tau_\cK(\eta)/t_K$.
Therefore, setting $\delta = \tilde{\delta}/t_K$, we get
\begin{align}
\esp[(x,i)]{\exp\PAR{ \delta \tau_\cK(\eta)}} \le
  \esp[(x,i)]{\exp\PAR{ \tilde{\delta}\tilde{\tau}_\cK}} &\leq
                                                           \tilde{C}W(x,i)
     = \tilde{C} \exp(\chi S(x)). \label{eq:tauEtaMajore}
\end{align}

\subsection{Proof of the main results}
\label{sec:prf_clt}

\begin{proof}[Proof of Theorem \ref{thmexperg}]
 We apply~\cite[Theorems 6.2]{DMT95}, choosing $f=1$. In our notation, this result
 tells us that 
 \[ W_0(x,i) = 1 + \frac{1}{\delta} \PAR{ \esp[(x,i)]{\exp(\delta \tau_\cK(\eta)) } - 1}\]
 is a Lyapunov function for the chain sampled at integer times. Using
 \cite[Theorem 5.2]{DMT95}, this implies that $(X_t,I_t)$ is $W_0$-uniformly
 ergodic, which implies the total variation bound~\eqref{eq:experg}, 
 since $W_0(x,i) \leq C_\delta W(x,i)$, for some positive constant $C_\delta$ by Equation~\eqref{eq:tauEtaMajore}. 
\end{proof}

\begin{proof}[Proof of Theorem~\ref{th:ergo}]
  We prove here that the ``exponential ergodicity'' given by the
  convergence in total variation indeed implies that ergodic means converge
  for reasonable test functions. This result seems to be folklore and similar
  results can be found in the literature, but we include a proof for completeness. 

  \textbf{Step 1. } Let us check that for any initial probability measure $\nu$, $\nu P_t$ converges to $\mu$ in total variation. Indeed, for any
  measurable test function $\varphi$  bounded by $\frac 1 2$,
  \begin{align*}
\left|\dE^\nu(\varphi(X_t,I_t)) - \int \varphi d\mu\right| &\leq \int
                                                            \left| \dE^{(x,i)}
                                                            (\varphi(X_t,I_t))
                                                            - \int
                                                            \varphi
                                                            d\mu\right|
                                                            d\nu(x,i)\\
&
\le \int
                                                            \left\|
                                                                       \cL \left( (X_t,I_t) | (X_0,I_0) = (x,i)\right) - \mu \right\|_{TV}
                                                            d\nu(x,i).
  \end{align*}
Thus, taking the supremum over $\varphi$,
  \begin{align*}
\left\|\nu P_t - \mu\right\|_{TV}&
\le \int
                                                            \left\|
                                                                       \cL \left( (X_t,I_t) | (X_0,I_0) = (x,i)\right) - \mu \right\|_{TV}
                                                            d\nu(x,i).
  \end{align*}
For fixed $(x,i)$, $  \left\|
                                                                       \cL
                                                                       \left(
                                                                         (X_t,I_t)
                                                                         |
                                                                         (X_0,I_0)
                                                                         =
                                                                         (x,i)\right)
                                                                       -
                                                                       \mu
                                                                     \right\|_{TV}$
  converges to    0   by Theorem~\ref{thmexperg}, and Lebesgue's
 Theorem thus yields the conclusion.

Notice that, as a consequence, the probability measure $\mu$ 
is the unique invariant measure  and thus the process $(X_t,I_t)$ is ergodic with respect to $\mu$.

    
  \textbf{Step 2. } Let $f\in L^1(\mu)$ and suppose that we start from
  the invariant measure $\mu$.
  The result then follows from classical arguments of ergodic theory, see for example
  \cite[Theorems 9.6, 9.8]{Kal02}. 

  \textbf{Step 3. } Let $\nu$ be any initial measure.
  By a result from \cite{Tho00} (namely, the fact that
  item (f) implies item (a) in  Theorem 6.4.1,  p.~205, and
  the discussion p.~213), there exists
  a coupling $((X_t,I_t),(Y_t,J_t))$ such that
  \begin{itemize}
  \item $(X_t,I_t)$ has the law of the process starting from $\nu$,
  \item $(Y_t,J_t)$ has the law of the process starting from $\mu$,
  \item The stopping time $T= \inf\{t: (X_t,I_t)= (Y_t,J_t)\}$ is
    almost surely finite, and $\prb{X_t = Y_t , \, \forall t\geq T} = 1$.
  \end{itemize}
    
  \textbf{Step 4. } Using the coupling introduced in Step $3$, we write
  \begin{align*}
    &\abs{ \frac{1}{t} \int_0^t f(X_s,I_s) ds - \int f(x,i) d\mu(x,i) } \\
    &\quad \leq
    \abs{ \frac{1}{t} \int_0^t f(X_s,I_s) ds - \frac{1}{t} \int_0^t f(Y_s,J_s) ds }
    +
    \abs{ \frac{1}{t} \int_0^t f(Y_s,J_s) ds - \int f(x,i) d\mu(x,i) }.
  \end{align*}
  The second term on the right hand side converges almost surely to $0$ by Step~$1$.
  The first term is bounded above  by
  \[ \frac{1}{t} \int_0^{T\wedge t} \abs{f(X_s,I_s) - f(Y_s,J_s)} ds.\]
  When $t$ goes to infinity, the integral converges to $\int_0^T \abs{f(X_s,I_s) - f(Y_s,J_s)} ds$
  which is finite almost surely since \eqref{eq:finite_along_trajectories} holds,
  so the whole first term converges to zero. This concludes
  the proof.
  \end{proof}

\begin{proof}[Proof of the central limit Theorem~\ref{thm:clt}]
  We first show that it is enough to prove the result for the process starting
  from its invariant measure.

Indeed, suppose that the initial measure $\nu$ and the test function $f$
are such that \eqref{eq:finite_along_trajectories} holds, that is,
$\prb[\nu]{\forall t\ge
    0,\,\int_0^t|f(X_s,I_s)|ds<\infty}=1$. Let $g:\dR\to\dR$ be a bounded
  Lipschitz continuous test function, with Lipschitz constant~$\Lip(g)$. Let
  $G\sim{\cal N}(0,1)$ be a standard Gaussian random variable. For any
  $t>t_0>0$ and $\alpha>0$, one has
\begin{align*}
  & \left|
    \esp[\nu]{g\left(\frac{1}{\sqrt{t}}\int_0^tf(X_s,I_s)ds\right)}
    - \esp{g\left(\sqrt{c_f}G\right)}
    \right|\\
  &\le \left|
    \esp[\nu]{g\left(\frac{1}{\sqrt{t}} \int_0^tf(X_s,I_s)ds\right)
    - g\left(\frac{1}{\sqrt{t}} \int_{t_0}^tf(X_s,I_s)ds\right)}\right|
    + \left|\dE_{\nu}-\dE_{\mu}\right|\left[
    g\left(\frac{1}{\sqrt{t}}\int_{t_0}^tf(X_s,I_s)ds\right)
    \right]\\
  &\quad  +\left|
    \esp[\mu]{g\left(\frac{1}{\sqrt{t}}\int_{t_0}^t f(X_s,I_s)ds\right)}
    - \esp{g\left(\sqrt{c_f}G\right)}
    \right| \\
  &\le \Lip(g)\alpha + 2\|g\|_\infty \prb[\nu]{\frac{1}{\sqrt{t}} \int_0^{t_0}|f(X_s,I_s)|ds \ge \alpha}
     + 2\|g\|_\infty\| \cL(X_{t_0},I_{t_0} | X_0,I_0 \sim \nu ) - \mu \|_{TV} \\
     &\quad  + \left|\esp[\mu]{g\left(\frac{1}{\sqrt{t}}\int_{t_0}^tf(X_s,I_s)ds\right)}-\esp{g\left(\sqrt{c_f}G\right)}\right|.
\end{align*}
Assume that the CLT holds when starting from the equilibrium measure $\mu$.
By Slutsky's Theorem, this implies that the last term converges to zero as $t$
goes to infinity. Therefore
\begin{align*} \limsup_t
  &\left|
    \esp[\nu]{g\left(\frac{1}{\sqrt{t}}\int_0^tf(X_s,I_s)ds\right)}
    - \esp{g\left(\sqrt{c_f}G\right)}
  \right| \\
 &\quad  \leq
   \Lip(g)\alpha + 2\|g\|_\infty \limsup_t \left(
   \prb[\nu]{\frac{1}{\sqrt{t}} \int_0^{t_0}|f(X_s,I_s)|ds \ge \alpha}
   \right)
     + 2\|g\|_\infty\| \cL(X_{t_0},I_{t_0} | X_0,I_0 \sim \nu ) - \mu \|_{TV}. 
\end{align*}
Since \eqref{eq:finite_along_trajectories} holds,
the $\limsup$ on the right hand side vanishes. We can then  let $t_0\to\infty$, using
the result from the second step of the proof of Theorem~\ref{th:ergo}, 
and finally $\alpha\to 0$ to conclude. To sum up, it is enough now to
show the CLT when we start from the equilibrium measure, that is when $\nu = \mu$.

We are going to deduce the central limit theorem under $\dP_\mu$ from~\cite[Theorem
18.5.3]{IL71} which is dedicated to strongly mixing stationary discrete
time sequences. Section 18.7 in this book is dedicated to continuous time
stationary processes but only states central limit theorems under the stronger
uniformly mixing condition even if its introduction explains that the extension
of all the discrete time results gives rise to no serious difficulty and is
left to the reader. 
Let for $n\in{\mathbb N}$, $Y_n=\int_n^{n+1}f(X_s,I_s)ds$. One has
\begin{equation}
   \esp[\mu]{|Y_n|^{2+\epsilon}}
   \le \esp[\mu]{\left(\int_n^{n+1}|f(X_s,I_s)|ds\right)^{2+\epsilon}}
   \le \int|f|^{2+\epsilon}d\mu<\infty\label{inty}
\end{equation}
by the invariance of $\mu$ and Jensen's inequality. 

For $l,m,n\ge 1$, let $A$ (resp. $B$) be a Borel subset of $\dR^{l+1}$ (resp. $\dR^{n+1}$). 
By the Markov property,
\begin{align*}
   \esp[\mu]{\ind{A}(Y_0,\hdots,Y_l)\ind{B}(Y_{l+m},\hdots,Y_{l+m+n})}
   & = \esp[\mu]{\ind{A}(Y_0,\hdots,Y_l)\esp[\mu]{\ind{B}(Y_{l+m},\hdots,Y_{l+m+n})|(X_s,I_s)_{s\in[0,l+1]}}}\\
   & = \esp[\mu]{\ind{A}(Y_0,\hdots,Y_l)\prb[(x,i)]{(Y_{m-1},\hdots,Y_{m+n-1})\in B}\bigg|_{(x,i)=(X_{l+1},I_{l+1})}}.
\end{align*}
Since by Theorem \ref{thmexperg} and invariance of $\mu$,
\[
  \left|\prb[(x,i)]{(Y_{m-1},\hdots,Y_{m+n-1})\in
    B}-\prb[\mu]{(Y_{l+m},\hdots,Y_{l+m+n})\in B}\right|\le C'_\delta W(x,i)
e^{-\delta (m-1)},
\]
we deduce that
\[
  \begin{split}
   \left|
     \esp[\mu]{\ind{A}(Y_0,\hdots,Y_l)\ind{B}(Y_{l+m},\hdots,Y_{l+m+n})}
     - \prb[\mu]{(Y_0,\hdots,Y_l)\in A}\prb[\mu]{(Y_{l+m},\hdots,Y_{l+m+n})\in B}
   \right| 
   \\
   \le C'_\delta e^{-\delta (m-1)}\int W d\mu,
 \end{split}
 \]
where $\int W d\mu<\infty$ by \eqref{eqborns}.
Therefore, under $\dP_\mu$, the stationary sequence $(Y_n)_{n\ge 0}$ is
strongly mixing in the sense of Definition 17.2.1 \cite{IL71} with mixing
coefficient $\alpha(m)=C'_\delta e^{-\delta (m-1)}\int W d\mu$. Since
$\sum_{n=1}^\infty\alpha(n)^{\epsilon/(2+\epsilon)}<\infty$, Theorem
18.5.3 \cite{IL71} ensures that $\sum_{n=1}^\infty|\esp[\mu]{Y_0Y_n}|<\infty$
and that, under $\dP_\mu$, 
\[
  \frac{1}{\sqrt{n}}\sum_{j=1}^n Y_j
  \text{ converges in distribution to }
  \cN\PAR{0,\esp[\mu]{Y_0^2}+2\sum_{n=1}^\infty\esp[\mu]{Y_0Y_n}}.
\] 
The remainder term $\int_{\lfloor t\rfloor}^t f(X_s,I_s)ds$ is easily 
dealt with.  Indeed, 
\(\left|\int_{\lfloor t\rfloor}^tf(X_s,I_s)ds\right|\le\int_{\lfloor t\rfloor}^{\lfloor
  t\rfloor+1}|f(X_s,I_s)|ds.\)
Under $\dP_\mu$, the right-hand is finite
according to \eqref{inty} and its law does not depend on $t$ which implies that
$\frac{1}{\sqrt{t}}\int_{\lfloor t\rfloor}^tf(X_s,I_s)ds$ converges in
probability to $0$ as $t\to\infty$. By Slutsky's theorem, we conclude that,
under $\dP_\mu$, 
\[
  \frac{1}{\sqrt{t}}\int_0^tf(X_s,I_s)ds \text{  converges in law to }
\cN\PAR{0,\esp[\mu]{Y_0^2}+2\sum_{n=1}^\infty\esp[\mu]{Y_0Y_n}}, \]
 as
$t\to\infty$.

To conclude the proof we now check that the limiting variance is
equal to $2\int_{0}^{\infty}\esp[\mu]{f(X_0,I_0)f(X_{t},I_{t})}dt$ where the
integral makes sense. First, combining a computation similar to the one leading
to the above strong mixing property with Theorem 17.2.2 \cite{IL71}, we obtain
that the conditions $\int |f|^{2+\epsilon}d\mu<\infty$ and $\int f d\mu=0$
ensure the existence of a finite constant $C$ such that for all $t\ge 0$,
$|\esp[\mu]{f(X_0,I_0)f(X_{t},I_{t})}|\le
Ce^{-\delta\epsilon t /(2+\epsilon)}$ so that
$\int_0^\infty|\esp[\mu]{f(X_0,I_0)f(X_{t},I_{t})}|dt<\infty$.
Moreover, by symmetry, invariance of $\mu$ and the change of variables $t=s-r$,
\begin{align*}
   \esp[\mu]{Y_0^2}
   =& 2\int_0^1\int_r^1 \esp[\mu]{f(X_r,I_r)f(X_s,I_s)}dsdr
    = 2\int_0^1\int_r^1\esp[\mu]{f(X_0,I_0)f(X_{s-r},I_{s-r})}dsdr\\
   =& 2\int_0^1\int_0^{1-r}\esp[\mu]{f(X_0,I_0)f(X_{t},I_{t})}dtdr
\end{align*} and for $j\ge 1$,
\begin{align*}
   \esp[\mu]{Y_0Y_j}
   =& \int_0^1\int_j^{j+1}\esp[\mu]{f(X_r,I_r)f(X_s,I_s)}dsdr
    = \int_0^1\int_j^{j+1}\esp[\mu]{f(X_0,I_0)f(X_{s-r},I_{s-r})}dsdr\\
   =&\int_0^1\int_{j-r}^{j+1-r}\esp[\mu]{f(X_0,I_0)f(X_{t},I_{t})}dtdr.
\end{align*}
As a consequence,
\begin{align*}
   \esp[\mu]{Y_0^2}+2\sum_{j=1}^\infty\esp[\mu]{Y_0Y_j}
   =& 2\int_0^1\int_{0}^{\infty}\esp[\mu]{f(X_0,I_0)f(X_{t},I_{t})}dtdr
    = \int_{0}^{\infty}\esp[\mu]{f(X_0,I_0)f(X_{t},I_{t})}dt. \qedhere
\end{align*}
\end{proof}

%
%
%

\subsection{Longtime convergence analysis on the auxiliary process \texorpdfstring{$\eta$}{eta} with generator \texorpdfstring{$H$}{H}}
\label{sec:auxiliary_process}
In order to use the Ray-Knight representation from Section~\ref{sec:RayKnight}
we need to get information on the Markov process $\eta$ with generator $H$
defined by~\eqref{eq:H}, since the steps in the 
$k \mapsto \Lambda_{j,r}(k)$ walk are defined in terms of $\eta$ (see
Equation~\eqref{eq:ray-knight-lambda} and Theorem~\ref{th:ray-knight}). 

The article \cite{TV11} already contains interesting bounds on
the convergence in total variation for $\eta$  to its invariant 
measure $\nu$ with density $Z^{-1} e^{-\frac{2}{\beta}\cosh(\beta x)}$ (\cite[Lemma 2.4]{TV11}),
but only for two specific initial distributions. Rather than 
adapting their proof, we provide a new one using a Lyapunov function. 
To get an idea for it, we make a guess at the behaviour of exponential 
moments of return times to the center of the space. If the starting
point $x$ of $\eta$ goes to $-\infty$, 
the process immediately jumps up to a bounded interval,
so we look
for a Lyapunov function $W^H$ which is bounded on $(-\infty,0]$. If
$x$ goes to $+\infty$, the process goes down at linear speed 
so we look for an exponential-like function $W^H$ at $+\infty$. 
\begin{lem}
  \label{lem:Vlyap}
  For any positive constant $s$,  let  $W^H_s$ be
 defined by by: 
  \[ W^H_s(x) = \begin{cases}
      2 & \text{ if } x\leq -s - 1 \\
      1 & \text{ if } -s \leq x \leq s \\
      e^{2 (x-s)} & \text { if } x\geq s
    \end{cases}
  \]
  in such a way that $W^H_s$ is smooth and with values in $[1,2]$ on the interval $[-s-1,-s]$.   Then for some $s$ large enough, $W^H_s$
  satisfies: for all $x \in \dR$,
  \begin{equation}
    \label{eq:Vlyap}
    H W^H_s(x) \leq - W^H_s(x) + C \ind{\abs{x} \leq c}
  \end{equation}
  for some positive constants $c$ and $C$, where $H$ is defined by~\eqref{eq:H}.
\end{lem}

\begin{thm}
  \label{thm:cvg_vt}
Let $(Q_t)_{t \ge 0}$ denote the transition semi-group associated with
the generator $H$  defined by~\eqref{eq:H}.
  For $W^H_s$ satisfying~\eqref{eq:Vlyap}, there exists $\alpha>0$ and
  $C<\infty$ such that for all $x \in \dR$,
  for all $t \ge 0$, 
   \[
     \NRM{\delta_xQ_t - \nu}_{TV} \leq C W^H_s(x)e^{-\alpha t},
   \]
where $\nu$ is the probability measure on $\dR$ with density $Z^{-1}
e^{-\frac{2}{\beta}\cosh(\beta x)}$ with respect to the Lebesgue measure.
\end{thm}
\begin{rem}
  Since $W^H_s(x)$ is bounded on $(-\infty,0)$, the process comes back from 
  $-\infty$ in finite time in the following sense:  $\forall \epsilon
  >0$, $\exists t_0>0$,
 $\forall x \le 0$, $\forall t \ge t_0$, $\NRM{\delta_xQ_t - \nu}_{TV}
 \leq \epsilon$.
 \end{rem}

Let us first prove Theorem~\ref{thm:cvg_vt} using Lemma~\ref{lem:Vlyap}. 
\begin{proof}[Proof of Theorem \ref{thm:cvg_vt}]
  This is a standard consequence of the existence of the Lyapunov
  function $W_s$. 
  To be more precise, in the terminology of \cite{MT09,DMT95}, the process is
 easily seen to be irreducible and aperiodic, and the 
  compact set $[-c,c]$ is petite (this can be seen as a consequence of
\cite[Proposition 4.2, Theorem 4.1]{MTStabilityII} and the fact that starting
from any point, the distribution of $\eta_1$ has a part which is absolutely continuous
with respect to Lebesgue measure).

  All the conditions are met to   apply \cite[Theorem 5.2 (c)]{DMT95}, which 
  implies the convergence in total variation claimed in Theorem~\ref{thm:cvg_vt}. 
\end{proof}
\begin{proof}[Proof of Lemma~\ref{lem:Vlyap}]
  Let $G$ be a random variable with density $e^{\beta y} \exp\left(-\frac{1}{\beta}\exp(\beta y)\right)$, that is, 
  $(-G)$ follows a Gumbel distribution. The generator $H$ may be rewritten as
  \begin{align*}
    Hf(x) =& -f'(x) + e^{-\beta x} \esp{ f(G) - f(x) | G\geq x} \\
          =& - f'(x) + \frac{e^{-\beta x}}{\prb{G\geq x}} \esp{ (f(G) - f(x)) \ind{G\geq x}}.
  \end{align*}
  Since $W^H_s(G)$ is integrable, $HW^H_s$ is locally bounded so it is enough to show 
  \(HW^H_s(x) \leq -  W^H_s(x)\)
  when $\abs{x}$ is  large. 
  
  \prfstep{Near $-\infty$}
  In this region 
  \begin{align*}
    HW^H_s(x) =& \frac{e^{-\beta x}}{\prb{G\geq x}} \esp{ (W^H_s(G) - 2)\ind{G\geq x}} \\
          =& e^{-\beta x} \PAR{\frac{\esp{W^H_s(G)\ind{G\geq x}}}{\prb{G\geq x}} - 2} \\
	  \leq& e^{-\beta x} \PAR{\frac{4}{3} \esp{W^H_s(G)} - 2}
	\end{align*}
	as soon as $\prb{G\geq x} \geq 3/4$. 
  Choose $s$ large enough to guarantee $\esp{W^H_s(G)} \leq 4/3$. Then the term between brackets
  is less than $-2/9$, and $HW^H_s(x)\leq -2 = -W^H_s(x)$ if $x$ is
  smaller than $-C$ for $C>0$ large  enough.

  \prfstep{Near $\infty$}
  Here the bound is mainly given by the drift term and we have to control the jump term. 
  \begin{align*}
    HW^H_s(x) = - 2 W^H_s(x) + \frac{e^{-\beta x}}{\prb{G\geq x}}e^{-2 s} 
    \esp{(e^{2 G} - e^{2 x}) \ind{G\geq x}}.
  \end{align*}
  We focus on the expectation in the second term. For $a>0$, 
  \begin{align*}
    \esp{(e^{2G} - e^{2 x}) \ind{G\geq x}}
    =& 
     \esp{(e^{2 G} - e^{2 x}) \ind{x\leq G< x+a}}
     + \esp{(e^{2 G} - e^{2 x}) \ind{x+a\leq G}} \\
    \leq& 
    e^{2 x} (e^{2 a} - 1) \prb{x\leq G}
     + \esp{e^{2 G} \ind{x+a\leq G}} \\
    \leq& e^{2 x} (e^{2 a} - 1) \prb{x\leq G}
     + \PAR{\esp{e^{4 G}}}^{1/2} \prb{x+a\leq G}^{1/2}
   \end{align*}
   where we used the Cauchy-Schwarz inequality on the last term. 
   Plugging this back in the previous equation we get
  \begin{align*}
    HW^H_s(x) \leq
    -  2 W^H_s(x) + e^{-\beta x}(e^{2a} -1) W^H_s(x) +
      \frac{e^{-\beta x - 2 s}\prb{G\geq x+a}^{1/2}}{\prb{G\geq x}} 
      \esp{e^{4 G}}^{1/2}. 
  \end{align*}
  Since $\prb{G\geq x} = \exp(-\frac{1}{\beta}e^{\beta x})$, $\prb{G\geq x+a}^{1/2} = \exp( - \frac{e^{\beta a} }{2\beta} e^{\beta x})$,
  so if $a>\ln(2)/\beta$, the last term goes to zero doubly exponentially fast when $x\to+\infty$.
  In the second term $e^{-\beta x}$ goes to zero, so we finally get
  \[HW^H_s(x) \leq -  W^H_s(x)\]
  when $x$ is large enough. 
\end{proof}

\subsection{Proof of Theorem~\ref{thm:Vdecreases}}
\label{sec:Vdecreases}

In order to prove Theorem~\ref{thm:Vdecreases}, we need an
intermediate lemma on the family of laws $q(y,z)\, dz$. Let us recall that
a probability measure $\mu_1$ on $\dR$ is said to be smaller for the stochastic
order than a probability measure $\mu_2$ on $\dR$, which we denote by
\[ \mu_1 \le_{s.o.} \mu_2\]
if and only if the associated survival functions are
ordered as follows:
\[
  \forall x \in \dR, \mu_1( (x, \infty)) \le \mu_2( (x,\infty)).
  \]
  Let us also recall the following classical property (see
  for example \cite[Proposition 17.A.2]{MOA11}): 
\begin{equation}\label{eq:so}
\mu_1 \le_{s.o.} \mu_2 \iff \text{for all increasing measurable
  bounded function } \varphi: \dR \to \dR, \int \varphi d\mu_1 \le
\int \varphi d\mu_2.
\end{equation}
\begin{lem}\label{lem:so}
Let us consider the family of probability measures $q(y,z)\,dz$
defined by~\eqref{eq:q}. Then, for any $y_1 \le y_2$, $q(y_1,z)\, dz
\le_{s.o.} q(y_2,z)\, dz$.
\end{lem} 
\begin{proof}
Let us first notice that for any $(x,y) \in \dR^2$,
\begin{align*}
\int_x^\infty q(y,z) \, dz
&= \int_{x \vee y}^\infty e^{\beta z} \exp\left(- \int_{y}^z e^{\beta
  s} ds \right) \, dz = \exp\left(- \int_{y}^{x \vee y} e^{\beta
  s} ds \right).
\end{align*}
We conclude by remarking that $\frac{\partial}{\partial y} \int_{y}^{x \vee y} e^{\beta
  s} ds = -1_{y < x}  e^{\beta
  y} \le 0$.
\end{proof}

We are now in position to prove Theorem~\ref{thm:Vdecreases}. Recall that $W(x,i) = \exp(\chi S(x))$, for a value 
of $\chi$ to be fixed at the end the proof. 

Let us consider the process $(X_t,I_t)$ with
generator~\eqref{eq:generator_metadynamics_flat} starting from $(X_0,I_0)=(x,i)$.
Let $(l_0,\ldots,l_K)$ be a discrete antiderivative of $x$
(see~\eqref{eq:antiderivative}, and let 
\[ \tilde{L}_t(j) = L_t(j) + l_j,\]
so that $j \mapsto L_t(j)$ is the time profile at time $t$  and $j
\mapsto \tilde{L}_t(j)$ is the 
potential in which the process $I_t$ evolves. 
Similarly let $\tilde{\Lambda}_{j,r}(k) = \Lambda_{T_{j,r}}(k) + l_k$: in this notation
the Ray-Knight equation~\eqref{eq:ray-knight-lambda} becomes
\begin{equation}
  \label{eq:ray-knight-lambda-tilde}
  \tilde{\Lambda}_{j,r}(k) = \tilde{\Lambda}_{j,r}(k-1) + \eta^-_k(\Lambda_{j,r}(k-1)). 
\end{equation}
Let
\[ 
  M_t = \max_j \tilde{L}_t(j) - \max_j \tilde{L}_0(j).
\]
It is easy to see that $M_t$ is also the total time $\int_0^t1_{\{\tilde{L}_s(I_s)=\max_{0\le j\le K}\tilde{L}_s(j)\}}ds$ between 
$0$ and $t$ that the process $(I_s)_{s \in [0,t]}$ spends at the current maximum of $\tilde{L}$. As a consequence,
\begin{equation}
   \label{majM}
\forall t\ge 0,\,M_t\le t.
\end{equation}
By the definition of $S$, $S(X_t)$ increases at rate $K-1$ when $\tilde{L}_t(I_t)=\max_{0\le j\le K}\tilde{L}_t(j)$ and decreases at rate $-1$ otherwise. As a consequence, $S(X_t) = S(x)-(t-M_t)+(K-1)M_t=S(x) - t + KM_t$, so that
\begin{align*}
  \esp[(x,i)]{W(X_t,I_t)}
  =& \esp[(x,i)]{\exp(\chi (S(x) - t + KM_t))} \\
  =& W(x,i)e^{-\chi t} \esp[(x,i)]{ e^{\chi K M_t}}.
\end{align*}
The main idea in the following is that, since the process is repelled
from the maxima of $\tilde{L}$, the fraction of the time $M_t/t$ it will spend
at the maxima before time $t$ should be way smaller than $1/K$, since
$A$ is sufficiently large so that
it will not have time to fill the large wells before time $t$.

We choose three parameters $(b,A,a)$ such that ($\alpha$ being the parameter introduced in Theorem~\ref{thm:cvg_vt})
\begin{align}
   &0<b<1/K\label{majb}\\
&1<A<1+\frac{bK}{2(K-1)}\label{minmajA}\\
&0<a<\frac{\alpha b}{6+2\alpha (K-1)} \text{ which implies } a < \frac{b}{K-1} \label{maja}
\end{align}
and write
\begin{align}
  \notag
  \esp[(x,i)]{W(X_t,I_t)}
  &= W(x,i)e^{-\chi t} \esp[(x,i)]{ e^{\chi K M_t}} \\
  \notag
  &\leq W(x,i)e^{-\chi t}\PAR{
     \prb[(x,i)]{M_t \leq bt} e^{Kb\chi t} + \prb[(x,i)]{M_t >bt}e^{K\chi t}} \\
   \label{eq:bound_on_exp_V}
  &\leq W(x,i)\PAR{
      e^{-\chi(1-Kb) t} + \prb[(x,i)]{M_t >bt}e^{K\chi t}},
\end{align}
where we used \eqref{majM} for the second inequality.
To continue, we must bound $\prb[(x,i)]{M_t \geq bt}$ from above. 
\begin{align*}
  \prb[(x,i)]{M_t > bt} =& \prb[(x,i)]{ \max_j \tilde{L}_t(j) > \max_j\tilde{L}_0(j) + bt} \\
  \leq& \sum_j \prb[(x,i)]{ \tilde{L}_t(j) > \max_k \tilde{L}_0(k) + bt} \\
  \leq& \sum_j \prb[(x,i)]{L_t (j) > \max_k l_k -l_j + bt}. 
  \end{align*}
If $j$ is not on a plateau, then introducing $\underline{j}=\max\{k\in\{1,\ldots,j\}:|x_k|>At\}$ and $\bar{j}=\min\{k\in\{j+1,\ldots,K\}:|x_k|>At\}$ with convention $\max\emptyset=0$ and $\min\emptyset=K+1$, then either $\underline{j}\ge 1$ and $x_{\underline{j}}<-At$ so that $l_{\underline{j}-1}-l_j=-\sum_{k=\underline{j}}^jx_k>At-\sum_{k=\underline{j}+1}^jx_k\ge (A-(K-1)a)t$ or $\bar{j}\le K$ and $x_{\bar{j}}>At$ so that $l_{\bar{j}}-l_j=\sum_{k=j+1}^{\bar{j}}x_k>At-(\bar{j}-(j+1))at\ge (A-(K-1)a)t $.
Hence $\max_k l_k-l_j+bt> (A+b-(K-1)a)t$, where the right-hand side is
larger than $t$ by the lower bound in \eqref{minmajA} and the second
inequality in~\eqref{maja}. So the probability is zero when $j$ is not on a plateau. Therefore it is enough to bound
\[
  \prb[(x,i)]{L_t(j)>bt}=\prb[(x,i)]{T_{j,bt}<t}
\]
when $j$ is on a plateau of $(x,i)$. By symmetry we may suppose
without loss of generality that 
there is a "cliff" on the right of $j$ (say between sites $r-1$ and
$r$, for $r > j$). 
In the notation defined above for the Ray-Knight process, 
we may bound the desired probability from above by 
\begin{align*}
  \prb[(x,i)]{ \text{ At time $T_{j,bt}$,  the total time spent in $r$ is less than $t$}}
  =&
  \prb[(x,i)]{ \Lambda_{j,bt}(r) < t}. 
\end{align*}
The fact that $\prb[(x,i)]{T_{j,bt}<t} \le \prb[(x,i)]{
  \Lambda_{j,bt}(r)  < t}$ is indeed obvious since $\Lambda_{j,s}(k)
\le T_{j,s}$ for any $(j,k) \in \{0, \ldots, K\}^2$ and
$s>0$.


By definition, $\Lambda_{j,bt}(j) = bt$. To use the Ray-Knight description we decompose
\begin{align}
  \notag
  \Lambda_{j,bt}(r)& = \sum_{j<k\leq r} \PAR{\Lambda_{j,bt}(k) - \Lambda_{j,bt}(k-1)} + \Lambda_{j,bt}(j) \\
  \notag
  &= \sum_{j<k\leq r}\PAR{ \tilde{\Lambda}_{j,bt}(k) - \tilde{\Lambda}_{j,bt}(k-1) - l_k + l_{k-1}} + bt \\
  \label{eq:decomposition_Lambda}
  & =  \sum_{j<k\leq r}\PAR{ \tilde{\Lambda}_{j,bt}(k) - \tilde{\Lambda}_{j,bt}(k-1)}
     - \sum_{j<k\leq r} x_k + bt. 
\end{align}
Since $x$ is $(t,a,A)$-separated and $j$ is on a plateau $\{l,\ldots,r-1\}$, we know 
by Definition~\ref{def:separated} that
$\ABS{x_k} \leq at$ for $j<k<r$, and $x_r<-At$. Therefore

\begin{align*}
  \Lambda_{j,bt}(r) 
  &\geq   \sum_{j<k\leq r}\PAR{ \tilde{\Lambda}_{j,bt}(k) - \tilde{\Lambda}_{j,bt}(k-1)}
     - (K-1)at + At  + bt \\
  &\geq  - \sum_{j<k\leq r}\ABS{ \tilde{\Lambda}_{j,bt}(k) - \tilde{\Lambda}_{j,bt}(k-1)}
       + At  
\end{align*}
since $b \ge (K-1)a$ (from~\eqref{maja}). Therefore 
\[ 
  \Lambda_{j,bt}(r) < t 
  \quad \implies \quad
  \exists k\in\{j+1,\ldots,r\}, \ABS{\tilde{\Lambda}_{j,bt}(k) - \tilde{\Lambda}_{j,bt}(k-1)} > \frac{A-1}{K} t.
\]
Let $D_k = \ABS{\tilde{\Lambda}_{j,bt}(k) - \tilde{\Lambda}_{j,bt}(k-1)}$. 
If the condition on the right is satisfied, picking the leftmost $k$
for which $D_k$ is larger than $\frac{A-1}{K} t$
leads to:
\[ 
  \prb[(x,i)]{\Lambda_{j,bt}(r) < t}
  \leq \sum_{j<k\leq r} \prb[(x,i)]{ D_k > \frac{A-1}{K} t 
    \text{ and } D_{k'} \le \frac{A-1}{K}t \text{ for all } j<k'<k}.
\]
Reusing the decomposition \eqref{eq:decomposition_Lambda} with $k-1$ in place of $r$, 
we can bound $\Lambda_{j,bt}(k-1)$ from below: on the event that the
$D_{k'}$ are smaller than $\frac{A-1}{K}t$ for $j<k'<k$,
\[ \Lambda_{j,bt}(k-1) \geq bt - (K-1) at - (K-1)\frac{A-1}{K}t \geq (b/2 - (K-1)a)t,\]
by the upper-bound in \eqref{minmajA}.
Therefore, by the Ray-Knight equation~\eqref{eq:ray-knight-lambda-tilde}
\begin{align*}
  &\prb[(x,i)]{ D_k > \frac{A-1}{K} t 
    \text{ and } D_{k'} \le \frac{A-1}{K}t \text{ for all } j<k'<k} \\
  &\leq \prb[(x,i)]{ \left| \eta^-_k(\Lambda_{j,bt}(k-1)) \right| > \frac{A-1}{K} t \text{ and } \Lambda_{j,bt}(k-1) \geq (b/2 - (K-1)a)t}
\\
  &=\esp[(x,i)]{\varphi(\Lambda_{j,bt}(k-1))1_{\{ \Lambda_{j,bt}(k-1) \geq (b/2 - (K-1)a)t\}}},
\end{align*}
where $\varphi(s)=\prb[(x,i)]{\left| \eta_k^-(s)\right| >
  \frac{A-1}{K}t }$ by independence of $\eta^-_k$ and
$\Lambda_{j,r}(k-1)$. Using Theorem~\ref{thm:cvg_vt} and the fact that $W_s(y)\le 2 e^{2\max(y,0)}$ for the second inequality, we obtain
\begin{align}
  \varphi(s)
  &\leq\esp[(x,i)]{\NRM{\delta_{\eta^-_k(0)} Q_s - \nu}_{TV}} +
    \nu\left(\left(-\infty, - \frac{A-1}{K}t\right) \cup \left(\frac{A-1}{K}t,\infty\right)\right) \notag\\
  &\leq Ce^{-\alpha s} \esp[(x,i)]{2e^{2\max(\eta^-_k(0),0)}}+  \nu\left(\left(-\infty, - \frac{A-1}{K}t\right) \cup \left(\frac{A-1}{K}t,\infty\right)\right).\label{majophi}
\end{align}
If $i<k$, then $\prb[(x,i)]{\eta_k^-(0)=x_k}=1$ and if $i\ge k$,
$\eta_k^-(0)$ is distributed according to $q(x_k,z) \, dz$ (see
Theorem~\ref{th:ray-knight}). In the latter case, since $x_k\leq at$
(if $k=r$, it is even negative) and the function $z\mapsto
e^{2\max(z,0)}$ is non-decreasing, using Lemma~\ref{lem:so}
and~\eqref{eq:so}, one has:
\begin{align*}
   \esp[(x,i)]{e^{2\max(\eta^-_k(0),0)}} = \int_{\dR} e^{2 \max(z,0)} q(x_k,z) \, dz \le
   \int_{\dR} e^{2 \max(z,0)} q(at,z) \, dz
\end{align*}
and thus
\begin{align*}
   \esp[(x,i)]{e^{2\max(\eta^-_k(0),0)}}\le \int_{at}^\infty
  e^{(2+\beta)z}e^{-\int_{at}^z e^{\beta s}ds}
  dz=e^{2at}+2\int_{at}^\infty e^{2z}e^{-\int_{at}^z e^{\beta s}ds}dz.
\end{align*}
Since $\int_{at}^\infty e^{2z}e^{-\int_{at}^z e^{\beta s}ds}dz \le
e^{-\beta at}\int_{at}^\infty e^{(2+\beta)z}e^{-\int_{at}^z e^{\beta
    s}ds}dz$ the second term in the right-hand side is negligible
compared to the first one in the limit
$t\to\infty$. Therefore, in both cases $i<k$ and $i\ge k$, $$\esp[(x,i)]{e^{2\max(\eta^-_k(0),0)}}\le C
e^{2at}\le Ce^{\alpha (b/2 - (K-1)a)t} e^{-at},$$ 
for some positive constant $C$, the second inequality being a
consequence of the condition \eqref{maja}
on $a$. Plugging this inequality into \eqref{majophi} an using the fact that the tails of $\nu$ are sub-exponential, we conclude that
\begin{align*}
  \forall s\geq (b/2 - (K-1)a)t,\,\varphi(s)\leq Ce^{-at}. 
\end{align*}

Retracing our steps, we see that
\begin{align*}
  \prb[(x,i)]{M_t\geq bt} &\leq \sum_{j \text{ on a plateau }\{l(j),\ldots,r(j)-1\}} \prb[(x,i)]{\Lambda_{j,bt}(r) \leq t} \\
  &\leq \tilde{C} e^{-at}
\end{align*}
for some positive $\tilde{C}$.
Inserting this bound in Equation~\eqref{eq:bound_on_exp_V}, and choosing $\chi=a/(2K)$, we  get
\begin{align*}
  \esp[(x,i)]{W(X_t,I_t)}
  &\leq W(x,i)\PAR{ e^{-\chi(1-Kb) t} + \tilde{C} e^{(K\chi - a) t}}
  \leq W(x,i)\PAR{ e^{-\chi(1-Kb) t} + \tilde{C} e^{- (a/2) t}}. 
\end{align*}
For some $t_0$ large enough, the term between brackets is strictly
less than $1$, as soon as $t \ge t_0$. 
Call this term $\kappa$.  We have proved that for the values of $a$, $A$, $\chi$, $\kappa$ 
and $t_0$ defined above, and for $t$ larger than $t_0$, 
$\esp[(x,i)]{W(X_t,I_t)} \leq \kappa W(x,i)$.
This concludes the proof of Theorem~\ref{thm:Vdecreases}.

\section*{Acknowledgements} This work is supported by the European Research Council under the European Union's Seventh Framework Programme (FP/2007-2013) / ERC Grant Agreement number 614492.

\printbibliography


\end{document}